\documentclass[12pt,thmsa]{amsart}
\usepackage{amsfonts}
\usepackage{amssymb, euscript}

\def\Cal{\mathcal}
\def\A{{\Cal A}}
\def\B{{\Cal B}}

\def\G{{\Cal G}}

\def\R{{\Cal R}}
\def\M{{\Cal M}}

\def\D{{\Cal D}}
\def\S{{\Cal S}}

\def\K{{\Cal K}}
\def\I{{\Cal I}}
\def\L{{\Cal L}}
\def\J{{\Cal J}}

\def\V{{\Cal V}}

\def\bbk{{\Bbb K}}

\def\bbr{{\Bbb R}}

\def\bbn{{\Bbb N}}
\def\bbh{{\Bbb H}}

\def\bbc{{\Bbb C}}

\def\diag{{\hbox{\rm diag}}}

\def\const{{\hbox{\rm const}}}
\def\cos{{\hbox{\rm cos}}}
\def\det{{\hbox{\rm det}}}

\def\Gr{{\hbox{\rm Gr}}}

\def\vol{{\hbox{\rm vol}}}

\def\rn{\bbr^n}

\def\part{\partial}

\def\b{\beta}

\def\lng{\langle}
\def\rng{\rangle}
\def\Gam{\Gamma}
\def\Om{\Omega}
\def\a{\alpha}
\def\om{\omega}

\def\Del{\Delta}
\def\del{\delta}
\def\vp{\varphi}

\def\gam{\gamma}
\def\Lam{\Lambda}
\def\sig{\sigma}
\def\lam{\lambda}

\def\e{\varepsilon}

\font\frak=eufm10

\def\fr#1{\hbox{\frak #1}}
\def\frA{\fr{A}}

\def\kn{\bbk^n}
\def\rn{\bbr^n}
\def\cn{\bbc^n}

\def\rN{\bbr^N}

\def\tgr4{\tilde {\rm Gr}_{4n,4n-4}}
\def\tgrs4{\tilde {\rm Gr'}_{4n,4n-4}}
\def\tg2{\tilde {\rm Gr}_{2n,2n-2}}

\def\sideremark#1{\ifvmode\leavevmode\fi\vadjust{\vbox to0pt{\vss
 \hbox to 0pt{\hskip\hsize\hskip1em
\vbox{\hsize2cm\tiny\raggedright\pretolerance10000
 \noindent #1\hfill}\hss}\vbox to8pt{\vfil}\vss}}}%

\newtheorem{theorem}{Theorem}[section]
\newtheorem{lemma}[theorem]{Lemma}
\newtheorem{definition}[theorem]{Definition}

\newtheorem{corollary}[theorem]{Corollary}
\newtheorem{proposition}[theorem]{Proposition}

\theoremstyle{remark}
\newtheorem{remark}[theorem]{Remark}

\numberwithin{equation}{section}

\newcommand{\be}{\begin{equation}}
\newcommand{\ee}{\end{equation}}

\newcommand{\bea}{\begin{eqnarray}}
\newcommand{\eea}{\end{eqnarray}}
\newcommand{\Bea}{\begin{eqnarray*}}
\newcommand{\Eea}{\end{eqnarray*}}

\begin{document}

\title[Comparison of volumes]
{Comparison of volumes of convex bodies in real, complex, and
quaternionic spaces}

\author{Boris Rubin}
\address{
Department of Mathematics, Louisiana State University, Baton Rouge,
LA, 70803 USA}

\email{borisr@math.lsu.edu}

\thanks{The  research was supported in part by the  NSF grant DMS-0556157.}





\begin{abstract}

The classical  Busemann-Petty problem (1956) asks, whether
origin-symmetric convex bodies in $\mathbb {R}^n$ with smaller
hyperplane central sections necessarily have smaller volumes. It is
known, that the answer is  affirmative if  $n\le 4$ and negative if
$n>4$. The same question can be asked when volumes of hyperplane
sections are replaced by  other comparison functions having
geometric meaning. We give unified exposition of this circle of
problems
  in   real, complex, and
quaternionic $n$-dimensional spaces. All cases are treated
simultaneously. In particular, we show that  the Busemann-Petty
problem in the
 quaternionic $n$-dimensional space  has an
affirmative answer if and only if $n =2$. The method relies on the
properties of cosine transforms on the unit sphere. Possible
generalizations  are discussed.

\end{abstract}

\maketitle

\section{Introduction}

\setcounter{equation}{0}

Real and complex affine and Euclidean spaces  are traditional
objects in integral geometry. Similar spaces can be  built over more
general algebras, in particular, over
 quaternions. The discovery  of quaternions is
 attributed to W.R.  Hamilton (1843).\footnote{As is mentioned by
 Truesdell \cite[p. 306]{Tru}, ``quaternions themselves were  first discovered, applied
 and published by Rodriges, Poisson's former pupil, in 1840".}
 A variety of problems of differential geometry in quaternionic  and more general
 spaces over algebras were investigated by  Rosenfel'd and
 his collaborators, in particular,  in the Kasan'  geometric school (Russia); see, e.g., \cite {Ros, VSS, Shi}.
Some problems of quaternionic integral geometry, mainly related
 to polytopes and
  invariant densities, were studied by  Coxeter,
 Cuypers, and others; see [Cu, GNT1,
 GNT2] and references
 therein.

In the present article we are focused on comparison problems for
convex bodies in the general context of the space
 $\bbk^n$, where $\bbk$ stands for the field $\bbr$ of real numbers, the field
 $\bbc$ of complex numbers, and  the skew field $\bbh$ of  real quaternions.
   Since $\bbh$ is not commutative, special consideration is
 needed in this case.

  Let, for instance, $K$ and $L$ be origin-symmetric convex
bodies in $\rn$ with
 section functions $$S_K (H)\!=\!vol_{n-1} (K\cap H) \quad \mbox{ and} \quad S_L (H)\!=\!vol_{n-1} (L\cap
 H),$$ $H$ being a
 hyperplane passing through the origin. Suppose that $S_K (H)\le S_L
 (H)$ for all such $H$. Does it follow that $vol_n (K) \le vol_n (L)$?
 Since the latter may not be true, another question arises: For which operator $\D$ is the
 implication  \be \label {sool}\D S_K (H)\le \D S_L (H)\quad \forall H \Longrightarrow vol_n(K) \le
 vol_n(L)\ee valid? Comparison problems of this kind
   attract   considerable attention in
the last decade, in particular, thanks to remarkable connections
with harmonic analysis. The first question is known as the
Busemann-Petty (BP) problem \cite {BP}. Many authors contributed to
its solution, e.g., Ball \cite {Ba}, Barthe,  Fradelizi, and Maurey
\cite {BFM}, Gardner \cite{Ga1, Ga2, GKS}, Giannopoulos \cite{Gi},
Grinberg and  Rivin \cite {GRi}, Hadwiger \cite{Ha}, Koldobsky
\cite{K}, Larman and Rogers
 \cite{LR}, Lutwak \cite {Lu}, Papadimitrakis \cite {Pa}, Rubin \cite {R5}, Zhang \cite{Z2}. The answer is really
striking. It is ``Yes''  if and only if $n \le 4$; see \cite{Ga3,
GKS, K, KY}, and references therein.

The second question, related to implication (\ref{sool}),  was asked
by Koldobsky, Yaskin, and  Yaskina  \cite{KYY}. It was called  the
{\it modified Busemann-Petty
 problem}.
Both questions were studied by Koldobsky,
  K\"onig, Zymonopoulou \cite{KKZ} and Zymonopoulou \cite{Zy} for convex bodies in
  $\bbc^n$. The answer to the first question for $\bbc^n$ is ``Yes" if and only if $n \le
  3$.

  We suggest unified exposition of these
  problems for   real, complex, and also
quaternionic $n$-dimensional spaces and the relevant
$(n-1)$-dimensional subspaces $H$. All these cases are treated
simultaneously. In particular, we show that {\it  the quaternionic
BP problem has an affirmative answer if and only if $n =2$.}

 The article is almost self-contained.  Our
proofs  essentially differ from those in the aforementioned
publications
 and rely on the properties of  the  generalized cosine transforms on the unit sphere \cite{R2,
R3, R7}.

The setting of the quaternionic BP problem  and its solution require
careful
 preparation and  new geometric concepts. The crux is  that, unlike
  the fields of real and complex numbers, the algebra
of quaternions is not commutative. This results in non-uniqueness of
quaternionic analogues of such concepts as a vector space and its
subspaces, a symmetric convex body, a norm, etc.

 Another motivation for our work is the
 lower dimensional Busemann-Petty problem (LDBP), which sounds like the usual BP
 problem, but the hyperplane sections  are replaced by plane
 sections of fixed dimension $1<i<n-1$. In the case $i=2, \; n=4$,  an affirmative
 answer to LDBP follows from the solution of the usual  BP problem. For
$i>3$, a
 negative answer was first given by Bourgain and  Zhang \cite{BZ}; see
 also \cite{K, RZ} for alternative proofs.  In the cases  $i=2$ and $i=3$ for $n>4$, the answer is
  generally unknown, however,  if the body with smaller sections is
 a body of revolution, the answer is
 affirmative; see  \cite{GZ}, \cite{Z1}, \cite{RZ}. The paper \cite{R8} contains a  solution of the LDBP problem
  in the more general situation, when the body with smaller sections is invariant
  under rotations, preserving mutually orthogonal
  subspaces of dimensions $\ell$ and $n-\ell$, respectively. The answer essentially depends on
   $\ell$.

   It is natural to ask,   {\it how  invariance properties
   of  bodies affect  the  corresponding LDBP problem}?

   Of   course, this question is too vague, however, every specific
   example might be of interest. The article \cite{KKZ} on the BP
   problem in $\bbc^n$ actually deals with the LDBP problem for
   $(2n-2)$-dimensional sections of $2n$-dimensional convex bodies, which are
   invariant under the block diagonal subgroup $G$ of $SO(2n)$ of the form
   $$
   G=\{ g=\diag(g_1, \ldots, g_n):\;  g_1= \ldots = g_n \in SO(2)\}.
   $$

  We will show that  the BP problem in the $n$-dimensional left and
   right quaternionic spaces $\bbh^n_l$ and $\bbh^n_r$ is equivalent to the LDBP problem for
   $(4n-4)$-dimensional sections of $4n$-dimensional convex bodies, which are
   invariant under a certain  subgroup $G\subset SO(4n)$ of block diagonal
   matrices, having $n$ equal $4\times 4$ isoclinic (or Clifford) blocks. Every
   such block is a left (or right) matrix representation of a real quaternion and has the property of rotating all
lines through the origin in $\bbr^4$ by the same angle. We give
complete solution to this ``$G$-invariant" comparison problem in the
general contest of $dn$-dimensional convex bodies, $n>1$, the
symmetry of which is determined by complete system of orthonormal
tangent vector fields on  the unit sphere  $S^{d-1}$.  The classical
result of differential topology says, that such systems are
available only on $S^{1}$, $S^{3}$, and $S^{7}$; see Section
 \ref{mson}. We also study the corresponding  modified BP, when the
``derivatives" $\D S_K$ and $\D S_L$ are compared.

\subsection* {Plan of the paper and main results}
 The significant part of the paper (Sections 2-4) deals with
 necessary preparations,  the aim of which is to make the text
 accessible to broad audience of analysts and geometers. In Sections 2.1 and 2.2
we recall basic facts about  quaternions  and vector spaces
$\bbk^n$, $\bbk \in \{\bbr, \bbc, \bbh\}$. This information is
scattered in the literature; see, e.g., \cite{KS, Lou, Por, Ta, Wo,
Z}. We present it in the form, which is suitable for our purposes.
Since $\bbh$ is not commutative, we have to distinguish the left
vector space $\bbh_l^n$ and the right vector space $\bbh_r^n$.

In   Section 2.3 we introduce a concept of {\it equilibrated body}
in the general context of the space $\frA^n$, where $\frA$ is a real
associative normed algebra. These bodies serve as a substitute for
the class of origin-symmetric convex bodies in $\bbr^n$. As in the
real case (see, e.g., \cite {Bar}), they are associated with norms
on $\frA^n$. In the complex case, other names (``absolutely convex"
or ``balanced") are also in use
 \cite{GL, Hou, Rob}. We could not find any description
of this class of bodies in the quaternionic or more general
 contexts and present this topic in detail.

In   Section 2.4 we give precise setting of the comparison problem
 of the BP type for equilibrated
 convex bodies in $\bbk^n\in \{\bbr^n,
\,\bbc^n, \,\bbh_l^n, \, \bbh_r^n\}$ (see Problem {\bf A}) and the
corresponding problems {\bf B} and {\bf C} for $G$-invariant convex
bodies in $\bbr^{N}$, $N=dn$. Here $d=1,2$, and $4$, which
corresponds to the real, complex, and quaternionic cases,
respectively. Section 2.5 contains necessary information about
vector fields on unit spheres and extends problems {\bf B} and {\bf
C} to the case  $d=8$. This value of $d$ cannot be increased in the
framework of  the problems {\bf B} and {\bf C}.

Section 3 provides the reader with necessary background from
harmonic analysis related to analytic families of cosine transforms
and intersection bodies.  The latter were introduced  by Lutwak
\cite{Lu} and generalized in different directions  ; see, e.g,
Gardner \cite{Ga3}, Goodey, Lutwak,  and Weil \cite{GLW}, Koldobsky
\cite {K}, Milman \cite{Mi}, Rubin and Zhang \cite {RZ}, Zhang \cite
{Z1}. Here we follow our previous papers \cite {R2, R3, R7}. We draw
attention to Section 3.2 devoted to homogeneous distributions and
Riesz fractional derivatives
 $D^\a=(-\Del)^{\a/2}$, where $\Del$ is the Laplace operator on
 $\rN$. An important feature of these operators is that the corresponding Fourier multiplier $|y|^\a$ does not
 preserve the Schwartz space $\S(\rN)$ and  the  phrases like  ``in
 the sense of distributions" (cf.  \cite {KYY, KY, Zy})  require careful
  explanation and justification.

Section 4  is devoted to weighted section functions of
origin-symmetric convex bodies. If $K$ is such a body, these
functions are
 defined as $i$-plane Radon transforms of the characteristic function $\chi_K
 (x)$, (i.e. $\chi_K
 (x)=1$ when $x\in K$, and $0$ otherwise) with integration against
 the weighted Lebesgue measure with a power weight $|x|^\b$. The
 usefulness of such functions was first  noted in \cite{R4} and
 mentioned
 in \cite[p. 492]{RZ}. Smoothness properties of these functions  play
 a decisive role in establishing main results, and we study them
 in detail. Similar properties in the context of the modified BP problem in $\bbr^{n}$
 and $\bbc^{n}$ were briefly
 indicated in  \cite {KYY, KY, Zy}, however, the details  (which are important and fairly nontrivial) were omitted.

 In Section 5 we obtain main results; see Theorems
 \ref{mcor}, \ref{37}, \ref{37b}, and Corollaries \ref{krrr}, \ref{mcors}.
 In particular, {\it the Busemann-Petty problem in $\kn$ has an affirmative
answer if and only if $n \le 2+2/d$, where  $d=1,2$, and $4$ in the
real, complex, and quaternionic cases, respectively.}

\noindent {\bf Acknowledgement.} I am grateful to  Professors Ralph
Howard, Daniel Sage,  and Michael Shapiro for useful  discussions.

 \noindent{\bf Notation.} We denote by   $
\sig_{n-1}\!=\! 2\pi^{n/2}/\Gam (n/2) $ the area of the unit sphere
$S^{n-1}$ in $\rn$; $SO(n)$  is the special orthogonal group. For
$\theta \!\in \!S^{n-1}$ and $\gam \!\in \!SO(n)$, $d\theta$ and
$d\gam$ denote the relevant probability measures; $\D(S^{n-1})$ is
the space of $C^\infty$-functions on $S^{n-1}$ with standard
topology; $ \D_e(S^{n-1})$ is the subspace of even functions in
$\D(S^{n-1})$.

In the following $M_{n,k}(\bbr)$ is the set of real matrices  having
$n$ rows and $k$
 columns; $M_n (\bbr)=M_{n,n}(\bbr)$; $A^T$ denotes  the transpose of a matrix $A$; $I_n \!\in\!
M_n (\bbr)$ is  the identity  matrix;  $V_{n, k}=\{F\in
M_{n,k}(\bbr):\; F^T F=I_k\}$ is the Stiefel manifold of orthonormal
$k$-frames in $\bbr^{n}$; $\Gr_k (V)$ is the Grassmann manifold of
$k$-dimensional linear subspaces of the vector space $V$.

 Given a
certain class $X$ of functions or bodies, we denote by $X^G$ the
corresponding subclass of $G$-invariant objects. For example, $C^G
(S^{n-1})$ and $D^G (S^{n-1})$ are the spaces of continuous and
infinitely differentiable functions on $S^{n-1}$, respectively, such
that $f(g\theta)=f(\theta) \; \forall g\in G, \; \theta \in
S^{n-1}$. An origin-symmetric (o.s.) star body in $\bbr^n$, $ n \ge
2$, is  a compact set $K$ with non-empty interior, such that $tK
\subset K \; \forall t\in [0,1]$, $K=-K$, and the {\it radial
function} $ \rho_K (\theta) = \sup \{ \lambda \ge 0: \, \lambda
\theta \in K \}$ is continuous on  $S^{n-1}$.
 We denote by $\K^n$
 the set of all o.s. star  bodies  in $\bbr^n$. A body $K\in \K^n$
 is said to be smooth if $\rho_K\in  \D_e(S^{n-1})$.

\section{Preliminaries}

\subsection{Quaternions}\label {mnok} We regard $\bbh$ as a normed
algebra over $\bbr$ generated by the units $e_0, \, e_1, \,e_2,
\,e_3$ (the more familiar notation is $\bf 1$, $\bf i$, $\bf j$,
$\bf k$, but we reserve these symbols for other purposes).  Every
element $q \in \bbh$
 is  expressed as $q=q_0e_0+q_1e_1+q_2e_2+q_3e_3$ ($ q_i \in
 \bbr$). We set
$$\bar
q=q_0e_0-q_1e_1-q_2e_2-q_3e_3, \qquad
|q|=\sqrt{q_0^2+q_1^2+q_2^2+q_3^2}.$$
 The multiplicative structure
in $\bbh$ is governed by the rules
$$
e_0 e_i=e_i e_0=e_i, \qquad i=0,1,2,3,
$$
$$
e_1 e_2=-e_2 e_1=e_3, \qquad e_2 e_3=-e_3 e_2= e_1,  \qquad  e_3
e_1=- e_1 e_3= e_2,
$$
$$
e_1^2=e_2^2=e_3^2=- e_0.
$$
The product of two quaternions $p=p_0e_0+p_1e_1+p_2e_2+p_3e_3$ and
$q=q_0e_0+q_1e_1+q_2e_2+q_3e_3$ is computed accordingly as \bea
pq&=&(p_0q_0-p_1q_1-p_2q_2-p_3q_3)e_0 \nonumber\\
&+&(p_0q_1+p_1q_0+p_2q_3-p_3q_2)e_1\label{pr1}\\
&+&(p_0q_2-p_1q_3+p_2q_0+p_3q_1)e_2\nonumber\\
&+&(p_0q_3+p_1q_2-p_2q_1+p_3q_0)e_3,\nonumber\eea   so that $$q\bar
q=\bar q q=|q|^2, \quad \overline{pq}=\bar q\bar p, \quad |pq| =
|p||q|, \quad q^{-1}=\bar q/|q|^2.$$ We identify
$$
\bbr=\{q \in \bbh: q_1=q_2=q_3=0\}, \qquad \bbc=\{q \in \bbh:
q_2=q_3=0\},$$ and denote by $Sp (1)$ the group of quaternions of
absolute value $1$. There is a canonical bijection $h: \bbh \to
\bbr^4$, according to which, \Bea q=q_0e_0+q_1e_1+q_2e_2+q_3e_3 \!
\!\!\!&&\overset {h}{\longrightarrow}\;  v_q=(q_0,q_1,q_2,q_3)^T,\\
Sp (1)&&\overset {h}{\longrightarrow} \; S^3. \Eea

 \noindent By (\ref{pr1}), \bea p\bar
q&=&(p_0q_0+p_1q_1+p_2q_2+p_3q_3)e_0
+(-p_0q_1+p_1q_0-p_2q_3+p_3q_2)e_1 \nonumber \\ &+&
(-p_0q_2+p_1q_3+p_2q_0-p_3q_1)e_2
+(-p_0q_3-p_1q_2+p_2q_1+p_3q_0)e_3,\nonumber \eea or \be\label {pr4}
p\bar q=\sum\limits_{i=0}^{3} (v_p \cdot A_i v_q) \,e_i,\ee \bea
A_0=I_4&=& \left[ \begin{array}{cccc}
1 & 0 & 0 & 0 \\
0 & 1 & 0 & 0 \\
0 & 0 & 1 & 0 \\
0 & 0 & 0 & 1 \end{array} \right ], \qquad A_1=\left[
\begin{array}{cccc}
0 & -1 & 0 & 0 \\
1 & 0 & 0 & 0 \\
0 & 0 & 0 & -1 \\
0 & 0 & 1 & 0 \end{array} \right ],\nonumber \\
\label{pr3p}{}\\
A_2&=&\left[
\begin{array}{cccc}
0 & 0 & -1 & 0 \\
0 & 0 & 0 & 1 \\
1 & 0 & 0 & 0 \\
0 & -1 & 0 & 0 \end{array} \right ],  \qquad  A_3=\left[
\begin{array}{cccc}
0 & 0 & 0 & -1 \\
0 & 0 & -1 & 0 \\
0 & 1 & 0 & 0 \\
1 & 0 & 0 & 0 \end{array} \right ].\nonumber  \eea Similarly,
\be\label {pr4s} \bar p q=\sum\limits_{i=0}^{3} (v_p \cdot A'_i
v_q)\, e_i,\ee  \bea &&A'_0=A_0=I_4, \qquad A'_1=\left[
\begin{array}{cccc}
0 & 1 & 0 & 0 \\
-1 & 0 & 0 & 0 \\
0 & 0 & 0 & -1 \\
0 & 0 & 1 & 0 \end{array} \right ],\nonumber \\
\label{pr3ps}{}\\
A'_2&=&\left[
\begin{array}{cccc}
0 & 0 & 1 & 0 \\
0 & 0 & 0 & 1 \\
-1 & 0 & 0 & 0 \\
0 & -1 & 0 & 0 \end{array} \right ],  \qquad  A'_3=\left[
\begin{array}{cccc}
0 & 0 & 0 & 1 \\
0 & 0 & -1 & 0 \\
0 & 1 & 0 & 0 \\
-1 & 0 & 0 & 0 \end{array} \right ].\nonumber  \eea One can readily
see that $A_i, A'_i \in SO(4) $ and  collections
$$\{A_0 v_q, A_1
v_q, A_2 v_q, A_3 v_q\}, \qquad\{A'_0 v_q, A'_1 v_q, A'_2 v_q, A'_3
v_q\}$$ form  orthonormal bases of $\bbr^4$ for every $q\in Sp (1)$
. This gives the following.
  \begin{theorem} \label {macr}There exist ``left rotations" $A_i$ and  ``right rotations"
 $A'_i \quad (i=1,2,3)$, such that for every $\sig\in S^3$, the frames
$$\{\sig, A_1
\sig, A_2 \sig, A_3\sig\}, \qquad\{\sig, A'_1 \sig, A'_2 \sig, A'_3
\sig\}$$ form  orthonormal bases of $\bbr^4$.
\end{theorem}

 The left- and right-multiplication mappings $p \to qp$ and $p \to
pq$  in $\bbh$ can be realized as linear transformations of
$\bbr^4$, namely, \be\label{m7} v_{qp}= L_q v_p, \qquad v_{pq}= R_q
v_p,\ee \be\label {f7} L_q\!=\!\left[
\begin{array}{cccc}
q_0 & -q_1 & -q_2 & -q_3 \\
q_1 & \, \,q_0 & -q_3 & \,\, q_2 \\
q_2 &  \,\, q_3 &  \,\, q_0 & -q_1 \\
q_3 & -q_2 &  \, \,q_1 &  \,\,q_0 \end{array} \right ], \quad
R_q\!=\!\left[
\begin{array}{cccc}
q_0 & -q_1 & -q_2 & -q_3 \\
q_1 &  \,\, q_0 &  \,\, q_3 & -q_2 \\
q_2 & -q_3 &  \, \,q_0 &  \,\, q_1 \\
q_3 &  \, \,q_2 & -q_1 &  \, \,q_0 \end{array} \right ].\ee
 These
formulas define regular representations of  $\bbh$ in the algebra
$M_4 (\bbr)$ of $4 \times 4$ real matrices: \be \rho_l: q \to L_q,
\qquad \rho_r: q \to R_{\bar q},\ee so that  $\rho_l (pq)=\rho_l (p)
\rho_l (q)$, $\;\rho_r (pq)=\rho_r (p)  \rho_r (q)$.
 Clearly,
\be \label{qbr} L_q=\sum\limits_{i=0}^3 q_i A_i, \qquad R_{\bar
q}=\sum\limits_{i=0}^3 q_i A'_i.\ee In particular, \be\label{ai}
A_i=L_{e_i}, \qquad A'_i=R_{\bar e_i},\qquad i=0,1,2,3.\ee For any
$p, q \in \bbh$, matrices $L_p$ and $R_q$ commute, that is,
\be\label{comu} L_pR_q =R_qL_p.\ee Moreover, $\det (L_q)=\det
(R_q)=|q|^4$ (see, e.g., \cite [p. 28]{Be}). Since the columns of
each of these matrices are mutually orthogonal, then, for $|q|=1$,
both matrices belong to $SO(4)$. The map
$$
Sp(1) \times Sp(1)  \longrightarrow SO(4), \qquad (p,q)
\longrightarrow L_pR_{\bar q},$$ is a group surjection with kernel
$\{(e_0, e_0),\;(-e_0, -e_0)\}$ \cite{Por, Wo}. A direct computation
shows that $x \cdot R_q x=x \cdot L_q x=q_0$ for every $x \in S^3$.
It means that both $L_q$ and $R_q$ have the property of rotating all
half-lines originating from $O$ through the same angle
$\cos^{-1}q_0$ (such rotations are called {\it isoclinic}  or {\it
Clifford translations} \cite {Wo}). We call $L_q$ and $R_q$ the {\it
left rotation} and the {\it right rotation}, respectively. Note also
that \be\label{gal} JL_q J=R_{\bar q}, \quad JR_q J=L_{\bar
q},\qquad J=\left[\begin{array}{cc}
-1 & 0 \\
0 &  I_3 \end{array} \right ].\ee It means that the left rotation
becomes the right one if we change the direction of the first
coordinate axis in  $\bbr^4$.

 Similarly, if $\bbk=\bbc$, we set \[\bbc \ni c=a +ib \overset
{h}{\longrightarrow} \;v_c=(a, b)^T \in \bbr^2,\]so that \be\label
{f8} v_{cd}=v_{dc}= M_c v_d; \qquad c,d \in \bbc,\qquad
M_c=\left[\begin{array}{cc}
a & -b \\
b &  a \end{array} \right ].\ee Clearly, $M_c \in SO(2)$ if $|c|=1$,
and, conversely, every element of $SO(2)$ has the form $M_c$,
$c=\cos \vp +i\sin \vp$.

\subsection {The space $\bbk^n$} Let $\bbk \in \{\bbr, \bbc, \bbh
\}$. Consider the set of ``points" $x=(x_1, \ldots, x_n), \; x_i \in
\bbk$, that can be regarded as an additive  abelian group in a usual
way. We want to equip this set with the structure of the inner
product vector  space over $\bbk$. The resulting space will be
denoted by $\kn$.  Unlike the cases $\bbk =\bbr$ and $\bbk =\bbc$,
in the non-commutative case $\bbk=\bbh$ it is necessary to
distinguish two types of vector spaces, namely, {\it right vector
spaces} and {\it left vector spaces}.

We recall (see, e.g., \cite {Art}) that an additive  abelian group
$X$ is a {\it right $\bbh$-vector space} if there is  a map $X
\times \bbh \longrightarrow X$, under which the image of each pair
$(x,q) \in X \times \bbh$ is denoted by $xq$, such that for all $q,
q', q'' \in \bbh$ and $x, x', x'' \in X$,

(a) $(x'+ x'')q=x'q+ x''q$;

(b) $ x(q' +q'') =xq' +xq''$;

(c) $x (q'q'')=(xq')q''$;

(d) $x e_0=x$.

\noindent Similarly,  an additive  abelian group $X$ is a {\it left
$\bbh$-vector space} if there is  a map $\bbh \times X
\longrightarrow X$, under which the image of each pair $(q,x) \in
\bbh \times X$ is denoted by $qx$, such that for all $q, q', q'' \in
\bbh$ and $x, x', x'' \in X$,

(a$'$) $q(x'+ x'')=qx'+ qx''$;

(b$'$) $ (q' +q'')x =q'x +q''x$;

(c$'$) $(q'q'')x =q'(q''x)$;

(d$'$) $e_0 x =x$.

According to these definitions, we define the left vector space
$\bbh^n_l$  to be the space of row vectors $x=(x_1, x_2, \ldots,
x_n), \quad x_j \in \bbh,$ with multiplication by scalars $c \in
\bbh$ from the left-hand side ($x  \to cx=(cx_1, cx_2, \ldots,
cx_n)$). We equip $\bbh^n_l$ with the {\it left inner product}
\be\label {i5} \langle x,y \rangle_l= \sum\limits_{j=1}^n  x_j  \bar
y_j.\ee The  right vector space $\bbh^n_r$  is defined as the space
of column vectors $x=(x_1, x_2, \ldots, x_n)^T, \quad x_j \in \bbh,$
with multiplication by scalars $c \in \bbh$ from the right-hand side
($x \to xc=(x_1 c, x_2 c, \ldots, x_n c)^T$) and with the {\it right
inner product} \be\label {i5s} \langle x,y \rangle_r=
\sum\limits_{j=1}^n \bar x_j y_j.\ee Clearly, $\overline { \langle
x,y \rangle_l }=\langle y,x \rangle_l$, $\;\overline { \langle x,y
\rangle_r }=\langle y,x \rangle_r$. Furthermore, if $x^*=(\bar
x)^T$, then
$$
 \langle x,y \rangle_l=\langle x^*,y^* \rangle_r, \qquad \langle x,y \rangle_r=\langle x^*,y^* \rangle_l.
 $$

 If $c$ is a real number, we can
write $cx=xc$  for both $x \in \bbh^n_l$ and $x \in \bbh^n_r$.
 If $\bbk=\bbc$ (or $\bbr$) we  regard  $\bbc^n$  (or $\bbr^n$) as the space of column vectors and set
\be\label {ibca} \langle x,y \rangle= \sum\limits_{j=1}^n \bar x_j
y_j,\ee as in (\ref{i5s}) (in the commutative case,  definitions
(\ref{i5}) and (\ref{i5s})  coincide up to conjugation: $\langle x,y
\rangle_l=\overline{\langle x,y \rangle_r}\,$).
\begin{definition} \label {elwz}  We write $\kn$ for  the vector
spaces $\bbh^n_l$, $\bbh^n_r$,  $\bbc^n$,  and $\bbr^n$,  equipped
with the inner product defined above. \end{definition}

There is a natural bijection $h: \kn \to \bbr^N$, $N=dn$, where $
d=1,2$, and $4$ in the real, complex, and quaternionic case,
respectively. Specifically,
 \be \label{bij}\bbh^n_l \ni x\!=\!
(x_1,  \ldots, x_n) \overset {h}{\longrightarrow}
v_x\!=\!\left[\begin{array}{c}
v_{x_1}  \\
...\\
v_{x_n}  \end{array} \right ]\in  \bbr^{4n},\ee

\be \label{bij1}\bbh^n_r \ni x\!=\!\left[\begin{array}{c}
x_1  \\
...\\
x_n  \end{array} \right ] \overset {h}{\longrightarrow}
v_x\!=\!\left[\begin{array}{c}
v_{x_1}  \\
...\\
v_{x_n}  \end{array} \right ]\in  \bbr^{4n},\ee

\be \label{bij2}\bbc^n \ni x\!=\! (x_1,  \ldots, x_n) \overset
{h}{\longrightarrow} v_x\!=\!\left[\begin{array}{c}
v_{x_1}  \\
...\\
v_{x_n}  \end{array} \right ]\in  \bbr^{2n},\ee where
$v_{x_i}=h(x_i)$. Abusing notation, we use the same letter $h$ for
both the scalar case, as in Section \ref{mnok},
 and the vector case, as in (\ref{bij})-(\ref{bij2}).

 Formulas (\ref {m7}) and (\ref {gal}) have obvious
extensions. Namely,

\vskip 0.1truecm

 \noindent for $ x \in (\bbh^n)_l$:
\be \label{mill} v_{qx}\!=\!\left[\begin{array}{c}
v_{qx_1}  \\
...\\
v_{qx_n}   \end{array} \right ]\!=\!\left[\begin{array}{c}
L_q v_{x_1}  \\
...\\
L_q v_{x_n}   \end{array} \right ]\!=\!\L_q v_{x},\quad
\L_q\!=\!\diag (L_q, \ldots , L_q);\ee

\noindent for $ x \in (\bbh^n)_r$: \be\label{milr}
v_{xq}\!=\!\left[\begin{array}{c}
v_{x_1q}  \\
...\\
v_{x_nq}   \end{array} \right ]\!=\!\left[\begin{array}{c}
R_q v_{x_1}  \\
...\\
R_q v_{x_n}   \end{array} \right ]\!=\!\R_q v_{x},\quad
\R_q\!=\!\diag (R_q, \ldots , R_q);\ee
 \be \label{219}\J \L_q \J=\R_{\bar q}, \quad \J \R_q \J=\L_{\bar q},\qquad \J=\diag (J,
\ldots , J).\ee Matrices $\L_q$, $\R_q$, and $\J$ have $n$ blocks;
 $\L_q$ and $\R_q$  belong to $ SO(4n)$, and $\J^2$ is the
identity matrix.

By (\ref {pr4}), the inner product (\ref{i5}) can be written as
\be\label{cb}
 \langle x,y  \rangle_l = \sum\limits_{i=0}^3 \langle x,y
 \rangle_i  \,e_i\ee
\be\label{ail} \langle x,y \rangle_i =v_x \cdot \A_i v_y, \qquad
\A_i=\diag (A_i, \ldots , A_i) \quad \mbox{\rm ($n$ blocks)},\ee
$A_i$ being defined by (\ref{pr3p}). Similarly, by (\ref {pr4s}),
\be\label{cbs}
 \langle x,y  \rangle_r = \sum\limits_{i=0}^3 \langle x,y
 \rangle'_i  \,e_i, \ee
\be\label{air} \langle x,y \rangle'_i =v_x \cdot \A'_i v_y, \qquad
\A'_i=\diag (A'_i, \ldots , A'_i).\ee By (\ref{ai}) and (\ref{219}),
\be\label{aii} \J \A_i \J= \A'_i, \qquad i=0,1,2,3.\ee

In the case $\bbk=\bbc$, for $x \in \bbc^n$ and $c \in \bbc$, owing
to (\ref{f8}), we have
 \be\label{milg} v_{cx}=v_{xc}=\M_c v_{x},\qquad \M_c=\diag (M_c, \ldots , M_c) \in SO(2n).\ee
Moreover, \be\label{vap} \langle x,y \rangle = v_x \cdot v_y -i (v_x
\cdot \B v_y), \ee \be\label{28} \B=\diag \left (
\left[\begin{array}{cc}
0 & -1  \\
1 & \,\,0  \end{array} \right ], \ldots , \left[\begin{array}{cc}
0 & -1  \\
1 & \,\,0  \end{array} \right ]\right ).\ee

We introduce the following block diagonal subgroups  consisting of
$n$ equal isoclinic blocks: \bea\label{bl1} \qquad G_{\bbh,
l}=&&\!\!\!\!\! \!\!\!\{g \in SO(4n): \\&&\!\!\!\!\! \!\!\!g=\L_q=
\diag (L_q, \ldots , L_q)\; \mbox {\rm for some}\; q \in \bbh, \;
|q|=1\}, \nonumber\eea
 \bea\label{bl2} \qquad G_{\bbh,
r}=&&\!\!\!\!\! \!\!\!\{g \in SO(4n):
\\&&\!\!\!\!\! \!\!\!g=\R_q= \diag (R_q, \ldots , R_q)\; \mbox {\rm
for some}\; q \in \bbh, \; |q|=1\}, \nonumber\eea

 \bea\label{bl3} \qquad  G_{\bbc}=&&\!\!\!\!\! \!\!\!\{g \in SO(2n):
\\&&\!\!\!\!\! \!\!\!g=\M_c=\diag (M_c, \ldots , M_c)\; \mbox {\rm
for some}\; c \in \bbc, \; |c|=1\}. \nonumber\eea

 If $\bbk=\bbr$,
then the corresponding group $G_{\bbr}$ consists of two elements,
namely, $I_n$ and $-I_n$. The groups  $G_{\bbh, l}$ and $ G_{\bbh,
r}$ are conjugate to each other by  involution $\J$: \be G_{\bbh,
l}=\J G_{\bbh, r} \J.\ee

\begin{definition} \label {ello} We will use   the unified notation $G$ for groups $G_{\bbh, l}$, $G_{\bbh,
r}$, $G_{\bbc}$, and $G_{\bbr}$.  \end{definition}

\subsection {Equilibrated convex bodies}
It is known that origin-symmetric convex bodies in $\rn$ are in
one-to-one correspondence with norms on $\rn$. {\it What is a
natural analogue of this class of bodies in  spaces over more
general fields  or algebras}?  Below we study this question  in the
general context of spaces over associative real normed algebras
$\frA$ with identity. Our consideration  generalizes
 the known reasoning for  real and complex numbers  \cite{Bar, GL, Hou, Rob}.

 We assume that $\frA$ contains real numbers and
denote by $|\lam|$ the norm of an element $\lam$ in $\frA$.
 Let $V$ be a left (or right) module over $\frA$. By relating vectors in $V$ new elements, called
points,  one obtains an affine space over $\frA$  \cite{Ros}. We
keep the same notation $V$ for this affine space.  As usual, a set
$A$ in $V$  is called convex if $x \in A$ and $y \in A$ implies $\a
x +\b y\in A$ for all $\a \ge 0,\, \b \ge 0,\, \a+\b=1$. A compact
convex set in $V$ with non-empty interior is called a {\it convex
body}.

 \begin{definition} A set $A$  in a left (right) space  $V$  over $\frA$ is called {\it equilibrated}  if
for all $x \!\in \!A$, $\lam x\! \in \!A$ ($x\lam \!\in \!A$)
whenever $\lam \!\in\! \frA$,$ \, |\lam|\!\le \!1$. \end{definition}

An equilibrated set in $\bbr^n$ is just an origin-symmetric
star-shaped set.
 The next definition agrees with
standard terminology for normed algebras; cf. \cite [p. 655]{Is}.
\begin{definition}\label{levd} Let $V$ be a left  space over  $\frA$. A function $p: V \to \bbr$ is called a norm if the
following conditions are satisfied:

{\rm (a)} $p(x) \ge 0$ for all $x \in V$;  $\;p(x) =0$ if and only
if $x=0$;

{\rm (b)} $p(\lam x)=|\lam |p(x)$  for all $x \in V$ and all $ \lam
\in \frA$;

{\rm (c)} $p(x+y)\le p(x) + p(y)$  for all $x,y \in V$.

\noindent If $V$ is a right  space over $\frA$, then {\rm (b)} is
replaced by

{\rm (b$'$)} $p(x\lam)=|\lam |p(x)$  for all $x \in V$ and all $
\lam \in \frA$.
\end{definition}

Let $V=\frA^n$ be the $n$-dimensional left (right) affine space over
$\frA$. Every point $x \in V$ is represented as $x=x_1 f_1 + \ldots
+x_n f_n$ ( $x=f_1 x_1 + \ldots +f_n x_n$), where $x_i \in \frA$ and
$f_1=(1,0, \ldots, 0), \ldots , f_n=(0,0, \ldots, 1)$
 is a standard basis in $V$. We set $||x||_2=(\sum\limits_{i=1}^n
 |x_i|^2)^{1/2}$.

\begin{lemma}\label{lev} Let $V=\frA^n$ be a left (right) space over $\frA$.

{\rm (i)} If $p: V\to \bbr$ is a norm, then \be A_p =\{x \in V: p(x)
\le 1\}\ee is an equilibrated convex body.

{\rm (ii)} Conversely, if $A$ is an equilibrated convex body in $V$,
then \be p_A (x)=||x||_A=\inf \{r>0: x \in rA \}\ee is a norm in $V$
such that $A=\{x \in V: ||x||_A \le 1\}$.
\end{lemma}

The proof of this lemma is standard and is given in Appendix.

In the following $\frA\equiv \bbk \in \{\bbr, \, \bbc, \,\bbh\}$;
$\;\bbk^n$ is any of the spaces $\bbr^n, \bbc^n, \bbh^n_l$ or
$\bbh^n_r$; $\;G \in \{G_{\bbr}, G_{\bbc}, G_{\bbh, l}, G_{\bbh, r}
\}$; see Definitions  \ref{elwz} and \ref{ello}; $\;N=n, 2n$, or
$4n$, respectively. Our next aim is to establish connection between
equilibrated convex bodies in $\kn$ and $G$-invariant
origin-symmetric star bodies in $\bbr^N=h(\kn)$. We recall the
notation \be\label{refl} \J=\diag \Big (\left[\begin{array}{cc}
-1 & 0 \\
0 &  I_3 \end{array} \right ], \ldots , \left[\begin{array}{cc}
-1 & 0 \\
0 &  I_3 \end{array} \right ]\Big ) \qquad \mbox{\rm ($n$
blocks)}.\ee Clearly, $\J$ acts on $\xi= (\xi_1, \xi_2, \dots ,
\xi_{4n})\in \bbr^{4n}$ by converting $\xi_1$ into $-\xi_1$, $\xi_5$
into $-\xi_5$, and so on.

\begin{theorem}\label{lle} Let $A$ be a set in $\kn$ and let $B=h(A)$ be its
image in $ \bbr^N$. Then

\noindent  {\rm (i)}   $A$ is convex if and only if $B$ is convex.

\noindent {\rm (ii)} $A$  is
 equilibrated in $\bbh^n_l$ if and only if  $B$ is $G_{\bbh, l}$-invariant and
 star-shaped.

\noindent {\rm (iii)} $A$  is
 equilibrated in $\bbh^n_r$ if and only if  $B$ is $G_{\bbh, r}$-invariant and
 star-shaped.

\noindent {\rm (iv)}  $A$  is
 equilibrated in $\cn$  if and only if  $B$  is $G_{\bbc}$-invariant
 and star-shaped.

\noindent {\rm (v)} $A$  is
 equilibrated in   $\rn$
  if and only if  it is origin-symmetric and star-shaped.

\noindent {\rm (vi)} A  set $S$ in $\bbr^{4n}$ is star-shaped and
$G_{\bbh, l}$-invariant (or $G_{\bbh, r}$-invariant) if and only if
the reflected set $\J S$ is star-shaped and  $G_{\bbh, r}$-invariant
($G_{\bbh, l}$-invariant, respectively).
\end{theorem}
\begin{proof} {\rm (i)} Since $h(\a x +\b y) \!= \!\a h(x) \! + \!\b h(y)$ for all
$\a, \b  \!\in  \!\bbr$ and $x,y  \!\in \! \kn$, then  $A$ and
$B=h(A)$ are convex simultaneously.

{\rm (ii)} Suppose that $A \subset \bbh^n_l$ is  equilibrated,
$\xi\in B$, and $x= h^{-1}(\xi)$. For any $q \in \bbh$ with $|q|=1$
we have $qx \in A$, and therefore,  $\L_q \xi =h(qx)\in B$.
Furthermore, for any $ \lam \in [0, 1]$, $\lam\xi=\lam h(x)=h(\lam
x)$. Since $\lam x \in A$, then $\lam \xi \in h(A)=B$. Thus, $B$ is
$G_{\bbh, l}$-invariant and star-shaped. Conversely, suppose that
$B=h(A)$ is star-shaped and  $G_{\bbh, l}$-invariant. Choose any $x
\in A,\; q \in \bbh, \; |q|\le 1$, and set $q=\lam \om, \; \lam
=|q|, \; |\om|=1$. We have
$$qx=\lam\om x=\lam h^{-1} h (\om x)=h^{-1} [\lam \L_\om h(x)].$$
Since $B=h(A)$ is $G_{\bbh, l}$-invariant, then $\L_\om h(x) \in B$
and since $B$ is star-shaped, then $\lam\L_\om h(x) \in B$. Hence,
$qx=h^{-1} [\lam \L_\om h(x)]\in A$.

The proof of {\rm (iii)} and {\rm (iv)}  follows the same lines with
obvious changes. The statement {\rm (v)} is trivial. The statement
{\rm (vi)} follows from (\ref{219}). Indeed, let  $S$ be a
star-shaped $G_{\bbh, l}$-invariant set in $\bbr^{4n}$  and let
$y\in \J S$. Then $y= \J x, \; x \in S$, and for any $q \in \bbh$
with $|q|=1$ we have $ \R_q y=\R_q \J x=\J \J \R_q \J x=\J\L_{\bar
q} x \in \J B$, because $\L_{\bar q} x \in B$. Furthermore, for any
$\lam \in [0,1]$, $\lam y=\lam \J x=\J \lam x\in \J B$, because
$\lam x \in B$. The reasoning in the opposite direction is similar.
\end{proof}

\subsection{Central hyperplanes in $\kn$ and $G$-invariant
Busemann-Petty problem in $\rN$}\label {3344}

Let $S_{\bbk^n} =\{y \in \kn: ||y||_2=1\}$ be the unit sphere in
$\kn$. Every hyperplane in $\kn$ passing through the origin has the
form \be\label{240} y^\perp =\{ x \in \kn: \langle x,y \rangle
=0\},\qquad y \in S_{\bbk^n},\ee where $\langle x,y \rangle$ is the
relevant inner product; see  (\ref{i5}), (\ref{i5s}), (\ref{ibca}).

 If $\bbk=\bbr$, this is a usual $(n-1)$-dimensional subspace
 of $\bbr^n$. If $\bbk=\bbc$, then, owing to
 (\ref{vap}), the equality $ \langle x,y \rangle
=0$ is equivalent to a system of two equations
$$
\xi \cdot \theta=0, \qquad \xi \cdot \B\theta=0,$$ where $ \xi=h(x)
\in \bbr^{2n}, \; \theta=h(y) \in S^{2n-1}$. This system can be
replaced by one matrix equation \be \label{2dp}F_2(\theta)^T \xi=0,
\qquad F_2(\theta)=[\theta, \B\theta] \in V_{2n, 2},\ee where
$V_{2n, 2}$ is the Stiefel manifold  of orthonormal $2$-frames in
$\bbr^{2n}$. Equation (\ref{2dp}) defines a $(2n-2)$-dimensional
subspace
 of $\bbr^{2n}$. The collection of all such subspaces will be
 denoted by $\Gr_{2n-2}^{\bbc} (\bbr^{2n})$.

In the non-commutative case $\bbk=\bbh$ we have two option. If
$\kn=\bbh^n_l$, then, owing to
 (\ref{cb}), the equality
$\langle x,y \rangle_l=0$  is equivalent to a system of four
equations
$$
\xi \cdot \A_i \theta =0 \qquad (i=0,1,2,3),
$$
or \be \label{4dp}F_{4,l}(\theta)^T \xi=0, \qquad
F_{4,l}(\theta)=[\A_0 \theta, \; \A_1 \theta, \; \A_2 \theta, \;
\A_3 \theta ]\in V_{4n, 4},\ee
 where $ \xi=h(x) \in \bbr^{4n}$, and $ \theta=h(y) \in S^{4n-1}$
(for simplicity, we use the same letters). If $\kn=\bbh^n_r$, then,
by (\ref{cbs}), $\langle x,y \rangle_r=0$  is equivalent to
 \be \label{4dps}F_{4,r}(\theta)^T \xi=0,  \qquad F_{4,r}(\theta)=[\A'_0 \theta, \; \A'_1
\theta, \; \A'_2 \theta, \; \A'_3 \theta ]\in V_{4n, 4}. \ee Since
$\A'_i
 =\J\A_i\J$ (see (\ref{aii})),  then \be\label {waas}
 F_{4,r}(\theta)\!=\!\J F_{4,l}(\J\theta)\quad \mbox{\rm for every}\quad \theta
\! \in \!S^{4n-1}.\ee Thus, (\ref{4dp}) and (\ref{4dps}) define two
different  $(4n-4)$-dimensional subspaces
 of $\bbr^{4n}$ generated by the same point $\theta \in S^{4n-1}$.  We
 denote by $\Gr_{4n-4}^{\bbh, l} (\bbr^{4n})$ and $\Gr_{4n-4}^{\bbh, r}
(\bbr^{4n})$ respective collections of all such subspaces, which
 are isomorphic to $S^{4n-1}$. By (\ref{waas}),
$$\Gr_{4n-4}^{\bbh, r}
(\bbr^{4n}) =\J \Gr_{4n-4}^{\bbh, l} (\bbr^{4n}).$$

Given $\theta \in S^{dn-1}\;$ ($d=1,2,4$), we will be using the
unified notation
 $H_\theta$ for the $(dn-d)$-dimensional subspace
orthogonal to   $ F_1(\theta)=\theta$, $F_2(\theta)$,
$F_{4,l}(\theta)$, and
 $F_{4,r}(\theta)$,  respectively.

 \begin{proposition}\label {t25}  The ``right"  manifold $\Gr_{4n-4}^{\bbh, r}
(\bbr^{4n})$ is invariant under the ``left"  rotations $\L_q $, that
is,
$$
 \L_q
\Gr_{4n-4}^{\bbh, r} (\bbr^{4n}) = \Gr_{4n-4}^{\bbh, r}
(\bbr^{4n}).$$ The ``left"  manifold $\Gr_{4n-4}^{\bbh, l}
(\bbr^{4n})$ is invariant under the ``right" rotations $\R_q $, that
is,
$$
 \R_q \Gr_{4n-4}^{\bbh, l} (\bbr^{4n}) = \Gr_{4n-4}^{\bbh, l} (\bbr^{4n}).$$
\end{proposition}
 \begin{proof}Let $H \in
\Gr_{4n-4}^{\bbh, r} (\bbr^{4n})$, that is,
 $H$ is orthogonal to $F_{4,r}(\theta)=[\A'_0 \theta, \; \A'_1
\theta, \; \A'_2 \theta, \; \A'_3 \theta ]$ for some $\theta \in
S^{4n-1}$.  Since $\L_p$ and $\R_q$ commute for any $p, q \in
 \bbh$ and $\A'_i=\R_{\bar e_i}$ (see (\ref{comu}) and (\ref{ai})), then $\L_q
 \A'_i=\A'_i \L_q$ and $\L_q F_{4,r}(\theta)=F_{4,r} (\L_q \theta)$. This
  implies  $$\L_q  \Gr_{4n-4}^{\bbh, r}
(\bbr^{4n}) \subset \Gr_{4n-4}^{\bbh, r} (\bbr^{4n})$$ for the
corresponding bundles of subspaces. By the same
 reason, we have $F_{4,r}(\theta)=\L_q F_{4,r}(\L_q^{-1} \theta)$ which gives the
 opposite embedding. The proof of  equality $\R_q \Gr_{4n-4}^{\bbh, l} (\bbr^{4n}) =
 \Gr_{4n-4}^{\bbh, l} (\bbr^{4n})$
 is similar.
\end{proof}

The above consideration enables us to give precise setting of the
Busemann-Petty problem in $\bbk^n$
  and reformulate the latter
  as the  equivalent lower dimensional
 problem for $G$-invariant convex bodies
 in $\rN$. We recall that $$N=dn; \quad n>1; \quad d=1,2,4; \quad G\in \{G_{\bbr}, \, G_{\bbc}, \, G_{\bbh, l}, G_{\bbh,
 r}\};$$ see (\ref{bl1}) - (\ref{bl3}). We will be using the unified notation $\tilde \Gr_{N-d}
 (\bbr^N) $ for the respective manifolds
 $$ \Gr_{n-1} (\bbr^n), \quad \Gr_{2n-2}^{\bbc} (\bbr^{2n}),\quad \Gr_{4n-4}^{\bbh, l}
 (\bbr^{4n}),\quad \Gr_{4n-4}^{\bbh, r} (\bbr^{4n})$$
 of $(N-d)$-dimensional
 subspaces $H_\theta$ introduced above.

\noindent {\bf Problem A.} {\it Let $A$ and $B$ be equilibrated
 convex bodies in  $\bbk^n$, $n>1$,  satisfying
 \be vol_{n-1} (A\cap \xi)\le vol_{n-1} (B\cap \xi)\ee for all central $\bbk$-hyperplanes $\xi$. Does it follow that
$vol_n (A) \le vol_n (B)$? }

Here volumes of geometric objects in $\bbk^n$ are defined as usual
volumes of their $h$-images in  $\rN$,  for example, $$vol_n
(A)=vol_{N} (h(A)), \qquad vol_{n-1} (A\cap \xi)=vol_{N-d} \,(h
(A\cap \xi)).$$

The equivalent lower dimensional problem is formulated as follows.

\noindent {\bf Problem B.} {\it Let $K$ and $L$ be  $G$-invariant
convex bodies in $\rN$,   with section functions

\centerline{$S_K (\theta)= vol_{N-d} (K \cap H_\theta), \quad$
$\quad S_L (\theta)= vol_{N-d} (L \cap H_\theta)$,}

 \noindent where $H_\theta \in \tilde \Gr_{N-d}
 (\bbr^N)$.  Suppose that $S_K (\theta)\le S_L
 (\theta)$  for all $\theta \in S^{N-1}$. Does it follow that $vol_N (K) \le
vol_N (L)$?}

We notice a fundamental difference between the usual LDBP problem,
where sections by {\it all} $(N-d)$-dimensional subspaces are
compared, and   Problem $B$, where, in the cases $d=2$ and $4$, the
essentially smaller (actually, $(N-1)$-dimensional) collection of
subspaces comes into play.

Since the  question in Problem {\bf B} may have a negative answer,
we also consider the following more general problem, which is of
independent interest.

\noindent {\bf Problem C.} {\it  For  which operator $\D$ does the
 assumption  $ \D S_K (\theta)\le \D S_L (\theta)\quad \forall \theta \in S^{N-1} $ imply
 $vol_N (K) \le vol_N (L)$? }

\subsection{Vector fields on spheres} \label {mson} Theorem \ref{macr}  suggests
 intriguing links between possible generalizations
of Problems {\bf B} and {\bf C} and the celebrated  vector field
 problem, which asks for the maximal
number $\rho (d)$ of orthonormal tangent vector fields on the unit
sphere $S^{d-1}$ in $\bbr^d$.

We recall some facts; see \cite {Hes, Hus, Ad}. A continuous tangent
vector field on $S^{d-1}$ is defined to be a continuous function $
\V: S^{d-1} \to \bbr^d$ such that $\V(\sig) \in \sig^\perp$  for
every $\sig \in S^{d-1}$. If $\V(\sig)=A\sig$, where $A$ is a $d
\times d$ matrix,  the vector field $\V$ is called linear. Vector
fields $\V_1, \ldots , \V_k$ on $S^{d-1}$ are called orthonormal if
for every $\sig \in S^{d-1}$, the corresponding
 vectors $\V_1(\sig), \ldots ,  \V_k (\sig)$ form an orthonormal frame
 in $\bbr^d$. The following result is known as the Hurwitz-Radon-Eckmann theorem
\cite{Hu, Rad, E}; see also \cite{Og}.
\begin{theorem} Let $d$ be a positive integer and write $d=2^{4s+r}t$, where $t$ is an odd integer,
 $s$ and $r$ are integers with $s\ge 0$ and $0\le r<4$. Then the
 maximal
number of  orthonormal  linear tangent  vector fields on  $S^{d-1}$
is equal to $\rho (d)=2^r+8s -1$.
\end{theorem}
 The number  $\rho
(d)$ is called the Radon-Hurwitz number. It is zero when $d$ is odd.
Adams \cite {Ad} extended this result to continuous
 vector fields. He proved that there  are at most $\rho (d)$ linearly
independent continuous
 tangent vector fields on  $S^{d-1}$.

In the case $\rho (d)=d-1$, when there exist a complete orthonormal
 system
 of linear tangent  vector fields $\{\V_1, \ldots , \V_{d-1}\}$ on
$S^{d-1}$,  the sphere
 $S^{d-1}$ is called {\it parallelizable}. The only parallelizable
 spheres are $S^1$, $S^3$, and $S^7$; see  Kervaire \cite{Ke},  Bott and
 Milnor \cite{BM}.

Complete systems of orthonormal linear tangent  vector fields on
$S^3$, namely,  $\{A_1 \sig, A_2 \sig, A_3\sig\}$ and $\{A'_1 \sig,
A'_2 \sig, A'_3 \sig\}$, where considered in Theorem \ref{macr}.
These  produce a series of new examples, for instance, \be\label
{lxxr} \{[\gam^{-1} A_1 \gam ] \sig, \;\;[\gam^{-1} A_2 \gam ]
\sig,\;\; [\gam^{-1} A_3 \gam ]\sig\}, \qquad \forall \gam \in
O(4).\ee

A complete system of orthonormal tangent linear vector fields on
$S^7$ can be constructed, e.g., as follows.

${}\quad  $If $\;\sig\;\;  =\; \; (\sig_1,\; \sig_2,\; \sig_3,\;
\sig_4,\; \sig_5,\; \sig_6,\; \sig_7,\; \sig_8)^T\in S^7$, then
\[
\begin{array}{rrrrrrrrrr}
A_1 \sig  & = & (\sig_2, & -\sig_1, &  \sig_4,  &  -\sig_3, & \sig_6, & -\sig_5, &  -\sig_8,  & \sig_7)^T, \\

A_2 \sig  & = & (\sig_3, & -\sig_4, &  -\sig_1, &  \sig_2, & \sig_7,  & \sig_8,  & -\sig_5,  & -\sig_6)^T, \\

A_3 \sig  & = & (\sig_4, &  \sig_3, &  -\sig_2, & -\sig_1, &  \sig_8, &   -\sig_7, &  \sig_6, & -\sig_5)^T, \\

A_4 \sig  & = & (\sig_5, &  -\sig_6, & -\sig_7, &
-\sig_8,&-\sig_1, &  \sig_2, &  \sig_3, & \sig_4)^T, \\

A_5 \sig  & = & (\sig_6, &  \sig_5, & -\sig_8,  &  \sig_7, &
-\sig_2, &  -\sig_1, & -\sig_4, &  \sig_3)^T, \\

A_6 \sig  & = & (\sig_7, &  \sig_8, &  \sig_5,  & -\sig_6, &
-\sig_3, &  \sig_4, & -\sig_1, &  -\sig_2)^T, \\

A_7 \sig  & = & (\sig_8, & -\sig_7, &   \sig_6, &  \sig_5, &
-\sig_4, & -\sig_3, & \sig_2, &  -\sig_1)^T.
\end{array}
\]
The corresponding matrices $A_i$, which are determined by
permutation of indices of coordinates $\sig_1, \ldots , \sig_8$ and
arrangements of $\pm$ signs, belong to $SO(8)$. More systems can be
constructed, e.g., as in (\ref{lxxr}).

The following statement can be found in \cite{Hes} in a slightly
more general form. For the sake of completeness, we present it with
proof.
\begin{lemma} \label {12as} ${}$\hfill

{\rm (i)} If $\sig\to A\sig$ is a linear tangent vector field on $
S^{d-1}$, then the $d\times d$ matrix $A$ is skew symmetric, that
is, $A+A^T=0$.

{\rm (ii)} If ${\bf A}\sig=\{A_i\sig\}_{i=1}^{d-1}$ is an
orthonormal system of linear tangent vector fields on $ S^{d-1}$,
then
$$
A_i^T A_j+ A_j^T A_i=0 \quad \mbox{\rm for all}\quad 1\le i<j\le
d-1,$$
$$
A_i^T A_i=I \quad \mbox{\rm for all}\quad 1\le i \le d-1.$$
\end{lemma}
\begin{proof} {\rm (i)} Let $\sig \cdot A\sig=0$ for all  $\sig\in
S^{d-1}$. Equivalently, $x \cdot Ax=0$ for all  $x\in \bbr^{d}$.
Then, for all $x,y \in \bbr^{d}$, \bea &&x \cdot (A+A^T)y=x \cdot Ay
+Ax \cdot y\nonumber\\&&= x \cdot Ax +x \cdot Ay+Ax \cdot y+Ay \cdot
y=(x+y)\cdot A(x+y)=0.\nonumber\eea Hence, $A+A^T=0$.

{\rm (ii)} As above, for all $x,y \in \bbr^{d}$ we have
$$
x\cdot (A_i^T A_j+ A_j^T A_i)y= A_i (x+y)\cdot A_j (x+y)=0,
$$
$$
x\cdot (A_i^T A_i -I)y=\frac{1}{2}\big [A_i (x+y)\cdot A_i (x+y)
-(x+y)\cdot (x+y)\big ]=0.
$$
This gives the result.
\end{proof}
\begin{lemma} Let ${\bf A}\sig=\{A_i\sig\}_{i=1}^{d-1}$ be
 an orthonormal system of linear tangent vector fields on
$ S^{d-1}$; $A_0=I$. Then \be \label {775} g_\lam ({\bf
A})\equiv\sum\limits_{i=0}^{d-1}\lam_i A_i\in O(d)\ee for every
$\lam=(\lam_1, \ldots , \lam_n)\in S^{d-1}$.
\end{lemma}
\begin{proof} By Lemma \ref{12as},
$$
g_\lam ({\bf A})^T g_\lam ({\bf A})=\Big
(\sum\limits_{i=0}^{d-1}\lam_i A_i^T \Big ) \Big
(\sum\limits_{j=0}^{d-1}\lam_j A_j \Big
)=\sum\limits_{i,j=0}^{d-1}\lam_i \lam_j A_i^T A_j=I.
$$
Hence, $g_\lam ({\bf A})\in O(d)$.
\end{proof}

Some notation are in order.
\begin{definition} \label {9981} Let $N=dn$, $d \in \{2,4,8\}, \; n>1$. Given an orthonormal system
 ${\bf A}\sig=\{A_i\sig\}_{i=1}^{d-1}$ of linear tangent vector fields
  on $ S^{d-1}$, we denote \bea \label {086s} \qquad \G_\lam ({\bf A})\!\!\!&=&\!\!\!\diag\, (g_\lam ({\bf A}), \,
  \ldots \, ,g_\lam ({\bf A}))\\
&=&\!\!\!\diag\, \Big (\sum\limits_{i=0}^{d-1}\lam_i A_i, \ldots ,
\sum\limits_{i=0}^{d-1}\lam_i A_i \Big ) \qquad \mbox{\rm ($n$ equal
blocks)},\nonumber\eea where $\lam \in S^{d-1}$. The corresponding
class of block diagonal orthogonal transformations of $\bbr^{N}$
(with $n$ equal $d\times d$ diagonal blocks), generated by ${\bf
A}$, is defined by \be  \label {247} G\!\equiv \!G(n,d; {\bf
A})\!=\!\{g\!\in \!O(N): \, g\!=\!\G_\lam ({\bf A}) \; \,\mbox{\rm
for some}\; \,\lam \!\in\! S^{d-1}\}.\ee We also introduce $N\times
N$ block diagonal matrices, containing $n$ blocks: \be \A_i=\diag
(A_i, \ldots , A_i) \qquad (i=1,2, \ldots , d-1), \ee and set
$\A_0=I_N$. Given $\theta \in S^{N-1}$, we denote by
 $H_\theta$  the $(N-d)$-dimensional subspace
orthogonal to  the $d$-frame \be \label {d453} F_d(\theta)=[\theta,
\A_1 \theta, \ldots ,  \A_{d-1} \theta ]\in V_{N,d}\ee
 and set
 \be \label {250}
\tilde \Gr_{N-d} (\bbr^N)=\{H_\theta:\; \theta \in S^{N-1}\}.\ee
\end{definition}

All objects in Definition \ref{9981} are familiar to us when $d=2,4$
(see Section \ref {3344}). Thus, Problems {\bf B} and {\bf C} extend
to the case $d=8$.

We recall  that the set $G$ of transformations  and  the set $\tilde
\Gr_{N-d} (\bbr^N)$ of planes  are determined by the orthonormal
system  ${\bf A}=\{A_i\}_{i=1}^{d-1}$ of  vector fields, which is
assumed to be fixed.

 The following lemma plays a crucial role in our consideration.
\begin{lemma}\label {L22} If $H \in \tilde \Gr_{N-d}
 (\bbr^N)$, then every continuous $G$-invariant function $f$ on $S^{N-1}$
  is constant on the $(d-1)$-dimensional section
$S^{N-1}\cap H^\perp$.
\end{lemma}
\begin{proof} Let
$H\equiv H_\theta$ be orthogonal to some $d$-frame (\ref{d453}). Any
point $\eta \in S^{N-1}\cap H^\perp$ is represented as
$$
\eta =\sum\limits_{i=0}^{d-1} \lam_i \A_i \theta, \qquad
\sum\limits_{i=0}^{d-1} \lam_i^2 =1,$$ or $\eta=\G_\lam ({\bf
A})\,\theta$; see (\ref{086s}). In particular, if $d=4$ and $A_i$
have the form (\ref{pr3p}), then  $\G_\lam ({\bf A})$ is a block
diagonal matrix with $n$ equal blocks of the form
$$
\sum\limits_{i=0}^3 \lam_i A_i=\left[
\begin{array}{cccc}
\lam_0 & -\lam_1 & -\lam_2 & -\lam_3 \\
\lam_1 &  \;\;\, \lam_0 & -\lam_3 &  \;\;\,\lam_2 \\
\lam_2 &  \;\;\,\lam_3 &  \;\;\,\lam_0 & -\lam_1 \\
\lam_3 & -\lam_2 &  \;\;\,\lam_1 &  \;\;\,\lam_0 \end{array} \right
]= L_\lam,$$
$$
\lam=\lam_0 e_0+\lam_1 e_1+\lam_2 e_2+\lam_3 e_3\in \bbh; \quad
\mbox{\rm (cf. (\ref{f7}))}.$$ Since $\G_\lam ({\bf A})\in G$, then
$f(\eta)=f(\G_\lam ({\bf A})\,\theta)=f(\theta)$. This gives the
result.
\end{proof}

\section{Cosine transforms and intersection bodies}

It is known \cite {R4, R5, R7, RZ} that diverse Busemann-Petty type
problems can be studied using analytic families of cosine transforms
on the unit sphere. This approach is parallel, in a sense, to the
Fourier transform method developed by Koldobsky and his
collaborators \cite {K, KY}. We shall see how these transforms can
be applied to   Problems {\bf A}, {\bf B}, and {\bf C} stated above.

 \subsection{Spherical Radon transforms and cosine transforms}
 We recall some basic facts; see \cite {R3, R7}. Fix an integer $i \in \{2,3, \ldots,
 N-1\}$ and let $\Gr_i(\bbr^N)$ be the Grassmann manifold of all
 $i$-dimensional linear subspaces $\xi$ of $\bbr^N$.
  The spherical Radon transform, that integrates
 a function $f \in L^1 (S^{N-1})$ over $(i-1)$-dimensional
 sections $S^{N-1}\cap \xi$, is defined by
 \be\label{rts}
 (R_i f)(\xi) = \int_{\theta \in S^{N-1}\cap\xi} f(\theta) \, d_\xi \theta,
\ee where $d_\xi \theta$  denotes the probability measures on
$S^{N-1}\cap\xi$. The  case $i=N-1$ in (\ref{rts}) is known as the
Minkowski-Funk transform \be\label{mf} (M f)(u)=\int_{\{\theta \,:
\, \theta \cdot u =0\}} f(\theta) \,d_u\theta=(R_{N-1} f)(u^\perp),
\qquad u \in S^{N-1}. \ee Transformation (\ref{rts}) can be regarded
 as a member (up to a multiplicative constant) of the analytic family
of the generalized cosine transforms
 \be\label{rka} (R_i^\a f)(\xi)=\gam_{N,i}(\a)\,
\int_{S^{N-1}} |\text{\rm Pr}_{\xi^\perp} \theta|^{\a+i-N} \,
f(\theta) \, d\theta,\ee
$$\gam_{N,i}(\a)\!=\!\frac{ \sig_{N-1}\,\Gamma((N\! - \!\a\!-\! i)/2)} {2\pi^{(N-1)/2} \,
\Gamma(\a/2)}, \quad Re \, \a \!>\!0, \quad \a+i-N \neq 0,2,4,
\ldots .$$ Here  $\text{\rm Pr}_{\xi^\perp} \theta $ stands for
 the  orthogonal
 projection of $\theta$ onto $\xi^\perp$.
If $f$ is smooth and $Re \, \a \le 0$, then $R_i^\a f$ is understood
as analytic continuation of integral (\ref{rka}), so that
\be\label{lim} \lim\limits_{\a \to 0} R_i^\a f=R_i^0 f =c_i \,R_if,
\qquad c_i=\frac{\sig_{i-1}}{2\pi^{(i-1)/2}}.\ee In the  case
$i=N-1$ we also set \be\label{af} (M^\a f)(u)=(R_{N-1}^\a
f)(u^\perp)= \gam_N(\a)\, \int_{S^{N-1}} f(\theta) |\theta \cdot
u|^{\a-1} \,d\theta, \ee \be\label{beren} \gam_N(\a)\!=\!{
\sig_{N-1}\,\Gamma\big( (1\!-\!\a)/2\big)\over 2\pi^{(N-1)/2} \Gamma
(\a/2)}, \qquad Re \, \a \!>\!0, \quad \a \!\neq \!1,3,5, \ldots
.\ee
\begin{lemma}\label{l1} \cite[Lemma 3.2]{R7} Let $\a, \b \in \bbc; \; \a, \b \neq
1,3,5, \ldots \,$. If $\a+\b=2-N$ and $f\in \D_e(S^{N-1})$ then
\be\label{st}M^\a M^{\b}f=f.\ee If $\a, 2\!-\!N\!-\!\a\! \neq\!
1,3,5, \ldots $, then $M^\a$ is an automorphism of $\D_e(S^{N-1})$.
\end{lemma}
\begin{corollary} The Minkowski-Funk transform on the space
$\D_e(S^{N-1})$
 can be inverted by the formula
\be\label{mmm} (M)^{-1}=c_{N-1}\,M^{2-N}, \qquad
c_{N-1}=\frac{\sig_{N-2}}{2\pi^{(N-2)/2}}.\ee
\end{corollary}

Both statements amount to Semyanisty \cite{Se2}, who used the
Fourier transform techniques. They can also  be obtained as
immediate consequence of the spherical harmonic decomposition of
$M^\a f$.
\begin{lemma}\label{l2} \cite[Lemma 3.5]{R7} Let $Re\, \a >0;\; \a \neq 1,3,5, \ldots \,
$. If $f \!\in \!L^1(S^{N-1})$, then \be \label{con} (R_i M^\a
f)(\xi)\!=\!c\,
 (R_{N-i}^{\a +i-1}
f)(\xi^\perp), \quad \xi \!\in \!\Gr_{i}(\rN),\quad c\!=\!
\frac{2\pi^{(i-1)/2}}{\sig_{i-1}},\ee \be \label{conn} (R_{N-i} M^\a
f)(\xi^\perp) = \frac{2\pi^{(N-i-1)/2}}{\sig_{N-i-1}} \, (R_i^{\a
+N-i-1} f)(\xi).\ee If
 $f \in \D_e(S^{N-1})$, then (\ref{con}) and (\ref{conn}) extend to $Re\, \a \le 0$ by
 analytic continuation.
\end{lemma}
\begin{proof} We sketch the proof for the sake of completeness. For $Re\, \a >0$,
\[(R_i M^\a f)(\xi)=\gam_N(\a) \int_{S^{N-1}\cap\xi} \, d_{\xi}
u  \int_{S^{N-1}} f(\theta) |\theta\cdot u|^{\a-1} \,d\theta.\]
Since $|\theta\cdot u|=|\text{\rm Pr}_{\xi} \theta||v_\theta \cdot
u|$ for some $v_\theta \in S^{N-1}\cap\xi$, changing the order of
integration, we obtain \[ (R_i M^\a f)(\xi)=\gam_N(\a)\,
\int_{S^{N-1}} f(\theta) |\text{\rm Pr}_{\xi} \theta|^{\a-1}
\,d\theta  \int_{S^{N-1}\cap\xi} |v_\theta \cdot u|^{\a-1} d_{\xi}
u.\] The inner integral is independent of $v_\theta$ and can be
easily evaluated.
 This gives (\ref{con}). Equality  (\ref{conn}) is a reformulation of (\ref{con}).
 \end{proof}

 An origin-symmetric star body $K$ in $\rN$ is completely determined by its {\it radial
function} $ \rho_K (\theta) = \sup \{ \lambda \ge 0: \, \lambda
\theta \in K \}$; see Notation.  Passing to polar coordinates, we
get \be\label{rraa}\vol_i(K\cap \xi) =\frac{\sig_{i-1}}{i}\,(R_i
\rho_K^i )(\xi), \qquad \xi \in \Gr_{i}(\rN).\ee

The next statement follows from Lemma \ref{L22} and plays the key
role in the whole paper.
\begin{lemma}\label{mal} Let $\rho_K \in D_e^G (S^{N-1})$, $N=dn$; $d\in \{1,2,4,8\}$,
$n>1$. Then for every subspace $H_\theta\in \tilde
\Gr_{N-d}(\bbr^N)$ with $\theta \in S^{N-1}$, \be\label{sub}
vol_{N-d} (K \!\cap \!H_\theta )\!=\!\frac {\pi^{N/2 -d}\,
\sig_{d-1}}{N-d}\, (M^{1-d}\rho_K^{N-d})(\theta).\ee
\end{lemma}
\begin{proof} Applying successively (\ref{rraa}) (with $k=N-d$), (\ref{lim}),
and (\ref{conn}) (with $\a=i+1-N, \;i=N-d$), we obtain \bea
vol_{N-d} (K \cap H_\theta )&=&\frac{\sig_{N-d-1}}{N-d} (R_{N-d}
\rho_K^{N-d})(H_\theta)\nonumber\\&=&\frac
{2\pi^{(N-d-1)/2}}{N-d}\,(R_{N-d}^0
\rho_K^{N-d})(H_\theta)\nonumber\\&=&\frac{\pi^{N/2-d}\,
\sig_{d-1}}{N-d} (R_{d}M^{1-d}
\rho_K^{N-d})(H_\theta^\perp).\nonumber\eea Since  $\rho_K$ is
$G$-invariant and $M^{1-d}$ commutes with orthogonal
transformations, then, by Lemma
 \ref {L22},  $M^{1-d} \rho_K^{N-d}\equiv \const$ on $S^{N-1}\cap
H_\theta^\perp$ and (\ref{sub}) follows.
\end{proof}

\begin{remark} In the classical  case $\bbk=\bbr$, when $N=n$ and
$d=1$, (\ref{sub}) becomes a particular case of  (\ref{rraa}):
$$  vol_{n-1} (K \cap\theta^\perp )=\frac{\sig_{n-2}}{n-1} (M
\rho_K^{n-1})(\theta),
$$
where $M$ is the Minkowski-Funk transform (\ref{mf}).
\end{remark}

\subsection{Homogeneous distributions and Riesz fractional
derivatives} Given  a  $G$-invariant infinitely smooth body  $K$ in
$\rN$ and a plane $ H_\theta \!\in\! \tilde \Gr_{N-d}(\bbr^N)$
generated by $\theta \in S^{N-1}$, we denote \be \label {secf} S_K
(\theta)=vol_{N-d} (K \cap H_\theta ). \ee

\noindent {\bf Question:} For which operator $A^\a$, \be A^\a
M^{1-d}\rho_K^{N-d}=(M^{1-\a}\rho_K^{N-d})(\theta) \,? \ee

\noindent  By (\ref{sub}), an answer to this question would give us
the corresponding equality for the section function \be A^\a
S_K(\theta) =c\, (M^{1-\a}\rho_K^{N-d})(\theta), \qquad c=\frac
{\pi^{N/2 -d}\, \sig_{d-1}}{N-d},\ee that paves the way to Problem
{\bf C}. By Lemma \ref{l1} we immediately get \be\label {dops} A^\a=
M^{1-\a}M^{1+d-N}.\ee To make this explicit formula  more
transparent and convenient to handle, we extend our functions by
homogeneity to the entire space $\rN$ and invoke powers of the
Laplacian. This idea was formally used in \cite {KYY, KKZ}, but it
requires justification and some  correction. Below we explain the
essence of the matter.

Let $\S(\bbr^N)$  be the  Schwartz
 space of  rapidly decreasing $C^\infty$  functions, and
 $\S'(\bbr^N)$   its dual.  The Fourier transform of  a distribution $F$
 in $\S'(\bbr^N)$ is defined by \footnote {\rm Here and on, the notation
$\lng \cdot, \cdot \rng$ and $(\cdot,\cdot)$ is used for
distributions on $\bbr^N$ and $S^{N-1}$, respectively.}
$$ \lng\hat F, \hat \phi\rng= (2\pi)^{N}\lng F, \phi\rng, \quad \hat \phi(y)=
\int_{\bbr^{N}} \phi(x) \, e^{ix \cdot y} \, dx, \quad \phi \in
\S(\bbr^N).$$  For $f \in L^1(S^{N-1})$, let $$(E_\lam
f)(x)=|x|^\lam f (x/|x|), \qquad x \in \bbr^N \setminus \{0\}.$$
This operator generates a meromorphic $\S'$-distribution, which is
defined by analytic continuation ($a.c.$) as follows:
$$ \lng E_\lam f, \phi \rng \!
=a.c. \int_0^\infty\! r^{\lam +N-1} u(r)\,dr, \quad
u(r)\!=\!\int_{S^{N-1}}
f(\theta)\,\overline{\phi(r\theta)}\,d\theta.$$
 The distribution $E_\lam f$ is regular if $Re
\,\lam >-N$ and admits simple poles at $\lam=-N, -N-1, \ldots $; see
\cite {GS}.
   If $f$ is orthogonal to all spherical
harmonics of degree $j$, then the derivative $u^{(j)}(r)$ equals
zero at $r=0$ and the pole at $\lam=-N-j$ is removable. In
particular, if  $f$ is  even, i.e., $(f, \vp)\!=\!(f, \vp_-), \;
\vp_- (\theta)\!=\!\vp (-\theta) \quad \forall \vp \!\in \!\D
(S^{N-1})$, then the only possible poles  of $E_\lam f$ are $-N,
-N-2, -N-4, \dots $.

Operator family $\{M^\a \}$ (see (\ref{af})) naturally arises thanks
to the formula \be \label{cf} [E_{1-N-\a} f]^\wedge=2^{1-\a}
\pi^{N/2}\,E_{\a-1}M^\a f,\qquad f\in \D_e(S^{N-1}),\ee which
  amounts to Semyanistyi \cite{Se2}. It holds pointwise for $0\!<\!Re \,
\a \!<\!1$ (see, e.g., Lemma 3.3 in \cite{R2} ) and extends in the
$S'$-sense to all $\a\in \bbc$ satisfying \be\label{alf}\a \notin
\{1,3, 5, \ldots \}
 \cup \{1-N, -N-1, -N-3, \ldots \}.\ee

The Riesz fractional derivative $D^\a \psi$ of order $\a \in \bbc$
of a Schwartz function $\psi$ is defined as a
$\S'(\bbr^N)$-distribution by the rule \be \label{lou}(2\pi)^{N}\lng
D^\a \psi, \phi\rng=\lng |y|^\a \hat \psi, \hat\phi\rng,\qquad \phi
\in \S(\bbr^N),\ee where the right hand side is a meromorphic
function of $\a$ with  simple poles $\a=-N, -N-2, \ldots$. One can
formally regard $D^\a$ as a power of minus Laplacian, i.e., $D^\a
=(-\Del)^{\a/2}$. The case of negative $Re \, \a$ corresponds to
Riesz potentials \cite{St}. Since multiplication by $|y|^\a$ does
not preserve the space $\S(\bbr^N)$, definition (\ref{lou}) is not
extendable  to arbitrary $\S'(\bbr^N)$-distributions.

To overcome this difficulty, Semyanistyi \cite{Se1} came up with the
brilliant idea to introduce
 another class of distributions as follows. Let $\Psi=\Psi (\bbr^N)$ be the subspace of
$\S(\bbr^N)$, consisting of functions $\om$ such that
$(\partial^\gam \om)(0)=0$ for all multi-indices  $\gam$.
 We denote by $\Phi=\Phi (\bbr^N)$  the Fourier image of $\Psi$, which
 is formed by Schwartz functions orthogonal to all polynomials. Let
$\Phi'$ and $\Psi'$ be the duals of $\Phi$ and $\Psi$, respectively.
 Two $\S'$-distributions, that coincide in the $\Phi'$-sense,
 differ from each other by a polynomial.
For any $\Phi'$-distribution $g$ and any $\a \in \bbc$, the Riesz
fractional derivative $D^\a g$ is correctly defined by the formula
\be \label {slizo}\lng D^\a g, \om\rng= (2\pi)^{-N}\lng \hat g,
|y|^\a \hat\om\rng,\qquad \om \in \Phi.\ee Clearly, multiplication
by $|y|^\a $ is a linear continuous operator on $\Psi$ (but not on
$\S$!); see \cite{R1, SKM} for  details and generalizations.
\begin{lemma}\label{mald} Let  $\a  \notin \{0, -2, -4, \ldots \}
 \cup \{N, N+2, N+4, \ldots \}$. If $f \in D_e (S^{N-1})$, then
 \be\label{tyi}
E_{-\a}M^{1-\a}f=2^{d-\a} D^{\a -d} E_{-d} M^{1-d}f \ee in the
$\Phi'$-sense. If, moreover,  $\a-d=2m$, $m=0,1,2, \ldots$, and
\be\label{dem} (D_m f) (\theta)=2^{-2m} [(-\Del)^m E_{-d}
f](x)|_{x=\theta},\ee then \be \label{orla} (M^{1-\a}f)(\theta)
=(D_m M^{1-d}f)(\theta)\ee pointwise for every $\theta \in S^{N-1}$.
\end{lemma}
\begin{proof} Replace $\a$ by $1-\a$ and by $1-d$ in (\ref{cf}). Denoting $c_\a=2^{-\a}
\pi^{-N/2}$ and $c_d=2^{-d} \pi^{-N/2}$, we get
$$
E_{-\a}M^{1-\a}f=c_\a \, [E_{\a-N}f]^\wedge, \qquad
E_{-d}M^{1-d}f=c_d \,[E_{d-N}f]^\wedge$$ (in the $\S'$-sense). Using
these formulas, for any test function $\om \in \Phi$ we obtain \bea
\lng E_{-\a}M^{1-\a}f, \om\rng&=& c_\a \lng [E_{\a-N}f]^\wedge, \om
\rng=c_\a\lng E_{\a-N}f, \hat\om \rng\nonumber\\&=& c_\a\lng
E_{d-N}f, |y|^{\a-d}\hat\om \rng=c_\a\lng E_{d-N}f,
[D^{\a-d}\om]^\wedge \rng\nonumber\\&=& c_\a\lng [E_{d-N}f]^\wedge,
D^{\a-d}\om \rng=c_\a c_d^{-1} \lng E_{-d}M^{1-d}f, D^{\a-d}\om
\rng\nonumber\\&=&2^{d-\a} \lng D^{\a-d}E_{-d}M^{1-d}f, \om
\rng.\nonumber\eea Let now $\a-d=2m$. Then $D^{\a-d}=(-\Del)^m $,
and the same reasoning is applicable for any $C^\infty$-function
supported in the neighborhood of the unit sphere. Hence, (\ref{tyi})
holds
 pointwise in this specific case, and (\ref{orla}) follows.
\end{proof}

Equalities  (\ref{sub}) and  (\ref{orla}) imply the following
\begin{corollary} \label{opya}
Let $S_K (\theta)$, $\theta \in S^{N-1}$, be a section function
(\ref{secf}) of a  $G$-invariant infinitely smooth body $K$ in
$\rN$; $N=dn$, $n>1$, $d\in \{1,2,4,8\}$. Let $D_m$ be a
differential operator (\ref{dem}), where $$2m \neq N-d, \,N-d+2, \,
N-d+4, \ldots \, .$$ Then \be (D_m S_K)(\theta) =c\,
(M^{1-d-2m}\rho_K^{N-d})(\theta), \qquad c=\frac {\pi^{N/2 -d}\,
\sig_{d-1}}{N-d}.\ee
  \end{corollary}

\subsection{Intersection bodies} We recall that $\K^N$ denotes
 the set of all origin-symmetric star  bodies  in $\bbr^N$. According
 to Lutwak \cite{Lu}, a body $K\in \K^N $ is called an intersection
body of a body $L\in \K^N $ if $ \rho_K (\theta) =\vol_{N-1}(L\cap
\theta^\perp)$ for every $\theta \in S^{N-1}$. A wider class of
intersection bodies, which is  the closure of the Lutwak's class in
the radial metric, was introduced by Goodey, Lutwak, and  Weil
\cite{GLW} as a collection of bodies $K\in \K^N $ with the property
$ \rho_K =M \mu$, where $M$ is the Minkowski-Funk transform
(\ref{mf}) and $\mu$ is an even nonnegative finite Borel measure on
 $S^{N-1}$.  The class of all such measures will be denoted by $\M_{e+}(S^{N-1})$.

 There exist several generalizations of the concept of  intersection
body \cite {K, Mi, R7, RZ, Z1}. One of  them relies on the fact that
the Minkowski-Funk transform $M$ is a member of the analytic family
$M^\a$ of the cosine transforms.

\begin{definition} \label{dk}\cite[Definition 5.1]{R7} For $0<\lam <N$,
 a body $K\in \K^N $ is called a $\lam$-intersection body
if there is a measure $\mu \in \M_{e+}(S^{N-1})$ such that
$\rho_K^\lam=M^{1-\lam}\mu $ $($by Lemma \ref{l1}, this is
equivalent to $M^{1+\lam -N}\rho_K^\lam\in \M_{e+}(S^{N-1}))$. We
denote by $\I^N_\lam$ the set of all such bodies.
\end{definition}

The equality $\rho_K^\lam=M^{1-\lam}\mu $ means that for any $\vp
\in \D(S^{N-1})$,
 $$
\int_{S^{N-1}} \rho_K^{k}(\theta)\vp (\theta)\, d\theta= \int_{S^{N-1}}
 (M^{1-\lam}\vp)(\theta)\, d\mu (\theta),$$
where for $\lam  \ge 1$, $(M^{1-\lam}\vp)(\theta)$ is understood in
the sense of analytic continuation.\footnote{There is a typo in
\cite {R7}: In Definition 5.1 and in the subsequent equality on p.
712 one should  replace $\rho_K$  by $\rho_K^\lam$.} If $\lam=k $ is
an integer, the class $\I^N_\lam$ coincides with Koldobsky's class
of $k$-intersection bodies  and agrees with his concept of isometric
embedding of the  space $(\bbr^N,||\cdot||_K)$ into
  $L_{-p}, \; p=\lam$ \cite{K}. In the framework of this concept,
  all bodies $K \in
  \I^N_\lam$ can be regarded as ``unit balls of $N$-dimensional
  subspaces of $L_{-\lam}$''.

The following  statement is a consequence of the trace theorem  for
 cosine transforms; see \cite[Theorem 5.13]{R7}.
\begin{theorem}\label{tyif} Let $1<m<N$, $\eta\in \Gr_m (\bbr^N)$, and let $0<\lam<m$.  If  $K\in
\I^N_\lam$ in $\bbr^N$, then $K\cap \eta \in \I^m_\lam$ in $\eta$.
\end{theorem}

 This fact
was  used (without proof) in \cite[Theorem 4]{KKZ}. In the case,
when $\lam=k$ is an integer, it was established by Milman \cite{Mi};
see \cite [Section 1.1]{R7} for the  discussion of this statement.

\section{Weighted section functions}

Let $K$ be an origin-symmetric convex body in $\bbr^N$. Given a
point $z \in int (K)$ (the interior of $K$), we define the shifted
radial function of $K$ with respect to $z$, \be\label{srf}
\rho(z,v)\!=\!\sup\{\lambda \!>\!0 : z\!+\!\lambda v\in K\}, \quad
(z,v) \in \Omega \!=\! int(K)\times S^{N-1}, \ee which is  a
distance from $z$ to the boundary of $K$ in the direction $v$.
\begin{lemma}\label{zhan} \cite[Lemma
3.1]{RZ} If an origin-symmetric convex body $K $ in $\bbr^N$
  has  $C^m$ boundary $\partial K$, $1\le m\le\infty$,
then $\rho(z,v)\in C^m (\Omega)$.
\end{lemma}
\begin{proof}  We recall the proof. Consider
the function
$$
v=g(z,x)=\frac{x-z}{\vert x-z\vert}, \qquad z\in int(K),\ x\in \partial K.
$$
Since $\partial K$ is $C^m$, $g(z,x)$ is a $C^m$ function in
$int(K)\times \partial K$. When $z$ is fixed, $g(z,\cdot)$ is a
$C^m$ diffeomorphism from $\partial K$ to $S^{N-1}$. By the implicit
function theorem, $x=f(z,v)$ is a $C^m$ function on $\Omega$. Thus,
$\rho(z,v)=\vert x-z \vert=\vert f(z,v)-z\vert$ is a $C^m$ function
on $\Omega$.\end{proof}

It was discovered by Gardner \cite{Ga1} and Zhang \cite{Z2}, that
positive solution to the Busemann-Petty problem for convex bodies
$K$ in $\Bbb R^3$ and $\Bbb R^4$ is intimately connected with the
volume of parallel hyperplane sections of those bodies; see also
\cite{K, KY}. This volume, which is a hyperplane Radon transform of
the characteristic function $\chi_K (x)$ of $K$,
 is represented as $A_{H, \theta}(t)=
vol_{N-1}(K\cap\{ H+t\theta\})$, where $t \in \bbr, \; \theta \in
S^{N-1}$, and $H$ is a hyperplane through the origin perpendicular
to $\theta$. It was noted in \cite{R4} and  in \cite[p. 492]{RZ},
that further progress can be achieved  if we replace $A_{H,
\theta}(t)$ by the mean value of  the $i$-plane Radon transform
\cite {He, R6} of some weighted function $f(x)=|x|^\b \chi_K (x)$.
This mean value should be taken
 over all $i$-planes parallel to a fixed subspace $\xi \in \Gr_i (\rN)$
at distance $|t|$ from the origin. Such averages for arbitrary  $f$
(see \cite[Definition 2.7]{R6}) play an important role in the theory
of $i$-plane Radon transforms. Similar ``weighted" section functions
were later used in \cite{KYY, Zy}.

Let us proceed with precise definition.  Given a convex body $K \in
\K^N$, we define the weighted section function \be \label {kpl}A_{i,
\b} (t,\xi)=\int_{ S^{N-1}\cap \xi^\perp} \Lam_\b (\xi + tu)\,
du,\qquad \xi\in \Gr_i (\bbr^N), \quad t\in \bbr, \ee where \be
\label{uyt}\Lam_\b (\xi + tu)= \int_{K\cap (\xi + tu)} |x|^\b\,
dx,\ee  is the $i$-plane Radon transform mentioned above. Clearly,
$A_{i, \b} (t,\xi)$  is an even function of $t$. Let $B=\{x: |x|\le
1\}$ be the unit ball in $\rN$ and let $ r_K\!=\!\sup \{ t\!>\!0: \;
tB \!\subset \!K \} $ be the radius of the inscribed ball in $K$.
\begin{lemma} \label {kya}If a convex body
 $K \!\in \! \K^N$ is  infinitely smooth and $\b \!> \!m \!- \!i$, then all derivatives
$$
A_{i, \b}^{(j)} (t,\xi)=\left ( \frac{d}{dt} \right )^j A_{i, \b} (t,\xi),
\qquad 0\le j\le m,
$$
are continuous in $(-r_K, r_K)\times \Gr_i (\bbr^N)$.
\end{lemma}
\begin{proof} Passing to polar coordinates in the plane $\xi + tu$, we
get \be \label {kogd}\Lam_\b (\xi + tu)=\int_{ S^{N-1}\cap \xi}
\!a^\b_{u,v} (t)\, dv, \ee $$  a^\b_{u,v} (t)= \int_0^{\rho(tu,v)}
r^{i-1} \,(r^2+t^2)^{\b/2}\, dr,$$ where $\rho(tu,v)$ is the radial
function (\ref{srf}).
 It suffices to show that for $\b>m-i$, all
  derivatives $(d/dt)^j a^\b_{u,v}
(t)$, $j=0,1, \dots, m$, are continuous on $(-r_K, r_K)$ uniformly
in $(u,v) \in (S^{N-1}\cap \xi^\perp) \times  (S^{N-1}\cap \xi)$.
 Let,
for short,  $\rho(t)\equiv \rho(tu,v)$. If $m=0$ and  $\b>-i$ the
uniform (in $u$ and $v$) continuity of $a^\b_{u,v} (t)$ follows from
Lemma \ref {zhan}. In the case $m=1$ we have $$(d/dt) a^\b_{u,v}
(t)= a_1 (t) +a_2 (t),$$ where $a_1 (t)=\rho^{i-1}
(\rho^2+t^2)^{\b/2} d\rho/dt$ is nice and $a_2 (t)=\b t
a^{\b-2}_{u,v} (t)$.  If $\b>2-i$ we are done. Otherwise,  if
$1-i<\b \le 2-i$, then \be\label{eda} a_2 (t)=\b t^{i+\b
-1}\int_0^{\rho/t} s^{i-1} (1+s^2)^{\b/2-1}\, ds \to 0, \quad
\mbox{\rm as} \quad t\to 0,\ee and the result is still true.
Continuing this process, we obtain the required result for all $m$.
\end{proof}

 The next lemma is a slight generalization of  the corresponding
statements in \cite {KYY} and \cite {Zy}.
\begin{lemma} \label {kyah} Let $K$ be an infinitely smooth origin-symmetric convex body in
$\bbr^N$, $\xi \in \Gr_i(\bbr^N), \; 1<i<N$.  If $-i<\b\le 0$, then
$A_{i, \b} (t,\xi)\le A_{i, \b} (0,\xi)$. If $\,2-i<\b\le 0$, then
$(d^2/dt^2)A_{i, \b} (t,\xi)|_{t=0}\le 0$.
\end{lemma}
\begin{proof} Replace $|x|^\b$
 in (\ref{uyt}) by $-\b\int_0^{1/|x|} z ^{-\b-1} dz$, $\b<0$, and change the order of integration.
 This gives
$$
\Lam_\b (\xi + tu)=-\b \int_0^\infty z^{-\b-1} \, vol_i ((B_{1/z}
\cap K) \cap (\xi + tu))\,dz$$ where $B_{1/z}$ is a ball of radius
$1/z$ centered at the origin. The  integral on the right hand side
is well defined if  $-i<\b<0$. Applying Brunn's theorem to the
convex body $B_{1/z} \cap K$, we obtain
$$
vol_i ((B_{1/z} \cap K)
\cap (\xi + tu))\le vol_i ((B_{1/z} \cap K) \cap \xi),$$ which
gives the first statement of the lemma. If $2-i<\b<0$, then, by
Lemma \ref{kya}, the
 derivative $(d^2/dt^2)A_{i, \b} (t,\xi)$ is continuous in the
 neighborhood of $t=0$ and  the second statement of the lemma follows from the first
 one. In the case $\b=0$ the result follows if we apply Brunn's theorem just to  $ K$.
\end{proof}

We recall some facts about analytic continuation ($a.c.$) of
integrals \be I(\a)=\frac{1}{\Gam (\a)} \int_0^\infty t^{\a -1}
f(t)\, dt, \qquad
 Re \, \a>0.\ee
\begin{lemma} \label {kryak} Let $m$ be a nonnegative integer, $f \in L^1 (\bbr)$.

\noindent{\rm (i)}  If, moreover, $f$  is $m$ times continuously
differentiable in the  neighborhood of $t=0$, then $I(\a)$ extends
analytically to $Re \, \a
>-m$. In particular,
for $-m < Re \, \a < -m  +1$, \be \label{hnhn1} a.c.\,
I(\a)=\frac{1}{\Gam (\a)} \int_0^\infty t^{\a -1} \Big
 [f(t)- \sum\limits_{j=0}^{m  -1} \frac{t^{j}}{j!} f^{(j)}(0)
\Big ] dt\ee and \be\label{netl} \lim\limits_{\a \to -m} I(\a)=
 (-1)^m f^{(m)}(0).\ee

\noindent{\rm (ii)} If $m$ is odd and $f$  is an even function,
which is $m+1$ times continuously differentiable in the
neighborhood of $t=0$, then (\ref{hnhn1}) holds for  $-m -1< Re \,
\a < -m  +1$.
\end{lemma}
\begin{proof} All statements are well known \cite{GS}. For instance,
{\rm (ii)} follows from the fact that all derivatives $f^{(j)} (t)$
of odd order are zero at $t=0$ and therefore, for $m$ odd,  the sum
$\sum_{j=0}^{m -1}$ can be replaced by
 $\sum_{j=0}^{m}$.
However, (\ref{netl}) is usually proved for functions, which have at
least $m+1$ continuous derivatives at  $t=0$. We show that it
suffices to have only $m$ continuous derivatives. The latter is
important in  our consideration. Let \bea (I^\lam
f)(t)&=&\frac{1}{\Gam (\lam)}\int_0^t f(s) (t-s)^{\lam -1}\,
dt\nonumber\\&=&\frac{t^\lam}{\Gam (\lam)}\int_0^1 f(t\eta)
(1-\eta)^{\lam -1}\, d\eta, \qquad \lam >0,\nonumber\eea  be the
 Riemann-Liouville fractional integral of $f$.
Note that
$$
f(t)- \sum\limits_{j=0}^{m -1} \frac{t^{j}}{j!} f^{(j)}(0)=(I^m
f^{(m)})(t)
$$
and $t^{-m}(I^m f^{(m)})(t)\to f^{(m)}(0)/m!$ as $t\to 0$. Hence,
for any $\e>0$ there exists $\del=\del (\e)>0$ such that
$$
|t^{-m}(I^m f^{(m)})(t)- f^{(m)}(0)/m!|<\e \qquad \forall t \in
(0,\del).$$ Setting $\a=\a_0 -m$, $\a_0 \in (0,1)$, we obtain

\bea &&\frac{1}{\Gam (\a)} \int_0^\infty t^{\a -1} \Big
 [f(t)- \sum\limits_{j=0}^{m  -1} \frac{t^{j}}{j!} f^{(j)}(0)
\Big ] dt- (-1)^m f^{(m)}(0)\nonumber \\
&&=\frac{1}{\Gam (\a_0-m)} \int_0^\del t^{\a_0 -1}\Big [t^{-m}(I^m
f^{(m)})(t)- f^{(m)}(0)/m!\Big]\, dt\nonumber \\
&&+f^{(m)}(0)\Big [\frac{\del^{\a_0}}{\a_0\,\Gam (\a_0-m)\, m!}
-(-1)^m \Big]\nonumber \\
&&+\frac{1}{\Gam (\a_0-m)} \int_\del^\infty t^{\a_0 -m -1} \Big
 [f(t)- \sum\limits_{j=0}^{m  -1} \frac{t^{j}}{j!} f^{(j)}(0)
\Big ] dt=I_1+I_2+I_3.\nonumber \eea If  $\a_0 \to 0$, then
$\a_0\,\Gam (\a_0-m)\, m! \to (-1)^m$,
$$
|I_1|<\frac{\e\,\del^{\a_0}}{\a_0\,|\Gam (\a_0-m)|} \to \e m!, \qquad I_2 \to 0, \qquad I_3 \to 0.
$$
This gives the result.
\end{proof}

The next lemma establishes connection between weighted section
functions,  spherical Radon transforms, and cosine transforms.
\begin{lemma} \label {krya} Let $\xi \in \Gr_i(\rN), \; 1<i<N$. Suppose that
$$\a \neq N-i, \,N-i+2, \,N-i+4, \ldots,$$ and  $K$ is an infinitely smooth
origin-symmetric convex body in $\rN$.

\noindent{\rm (i)} If $\b>-i$ and $Re \, \a
>0$,
then  \be \label{44}\frac{1}{\Gam (\a/2)}\int_0^\infty t^{\a-1}
A_{i, \b} (t,\xi)\, dt=c\, (R_{N-i}M^{\a+1+i-N}\rho_K
^{\a+\b+i})(\xi^\perp),\ee
$$
c=\frac{\pi^{i/2}\, \sig_{N-i-1}}{(\a+\b+i)\, \sig_{N-1}\, \Gam
((N-i-\a)/2)}.
$$

 \noindent{\rm (ii)}  If $\b>1-i$, then
(\ref{44}) extends
 to  $-1 < Re \, \a < 0$ as  \bea \label{hnhn}&&\frac{1}{\Gam
(\a/2)} \int_0^\infty t^{\a -1}
 [A_{i, \b}(t,\xi)\!-\! A_{i, \b}(0,\xi) ] dt\\&&=c\, (R_{N-i}M^{\a+1+i-N}\rho_K ^{\a+\b+i})(\xi^\perp).
\nonumber\eea

\noindent{\rm (iii)}  If $\b\!\ge\! 2\!-\!i$, then (\ref{hnhn})
holds in the extended domain $-2\!< \!Re \, \a \!< \!0$.

\noindent{\rm (iv)} If $\b>m-i$ and $m\ge 0$ is even, then  \be
\label{46}\frac{\Gam ((1-m)/2)}{2^{m+1}\, \sqrt {\pi}}\,A_{i,
\b}^{(m)}(0, \xi)=c_1 \, (R_{N-i}M^{1-m+i-N}\rho_K
^{\b-m+i})(\xi^\perp),\ee
$$
c_1=\frac{\pi^{i/2}\, \sig_{N-i-1}}{(\b-m+i)\, \sig_{N-1}\, \Gam
((N-i+m)/2)}.
$$
\end{lemma}
\begin{proof}  {\rm (i)} Consider the integral
\be\label{3hb} g_{\a,\b} (\xi)=\frac{1}{\Gam (\a/2)}\int_K
|P_{\xi^\perp} x |^{\a +i-N}\,|x|^\b\, dx, \qquad Re \, \a
>0,\ee
where $ P_{\xi^\perp}$ denotes the orthogonal projection onto
$\xi^\perp$. We transform (\ref{3hb}) in two different ways (a
similar trick was used in \cite[p. 61]{R5} and \cite [p. 490]{RZ}).
On the one hand, integration over slices parallel to $\xi$ gives
\bea
 g_{\a,\b} (\xi)&=&\frac{1}{\Gam
(\a/2)}\int_{\xi^\perp} |y|^{\a +i-N} \, dy \int_{K\cap (\xi + y)}
|x|^\b\, dx
\nonumber\\
&=&\label{3hp}\frac{1}{\Gam (\a/2)}\int_0^\infty t^{\a-1} A_{i, \b}
(t,\xi)\, dt.\eea On the other hand, passing to polar coordinates,
we can express $g_{\a,\b}$ as the generalized cosine transform
(\ref{rka}), namely,
 \bea
g_{\a,\b} (\xi)&=&\frac{1}{(\a+\b+i)\,\Gam (\a/2)}\int_{S^{N-1}}
\rho_K (u)^{\a+\b+i}\,  |P_{\xi^\perp} u |^{\a +i-N}\,
du\nonumber\\
&=& c_{\a,\b} (R^\a_{i}\rho_K^{\a+\b+i})(\xi),\nonumber\eea
$$
c_{\a,\b} =\frac{2\pi^{(N-1)/2}}{(\a+\b+i)\, \sig_{N-1}\, \Gam
((N-i-\a)/2)}.
$$
Hence, by (\ref{con}), \be g_{\a,\b}
(\xi)=\frac{c_{\a,\b}\,\sig_{N-i-1}}{2\pi^{(N-i-1)/2} }\, (R_{N-i}
M^{\a+1+i-N}\rho_K^{\a+\b+i})(\xi^\perp),\ee  which gives
(\ref{44}).

{\rm (ii)} By Lemma \ref{kya} (with $m=1$) the derivative $(d/dt)
A_{i, \b} (t,\xi)$ is continuous in the neighborhood of $t=0$.
Keeping in mind that
$$\lim\limits_{\a \to -m} \frac{\Gam (\a)}{\Gam (\a/2)}=\frac{\Gam
((1-m)/2)}{2^{m+1}\, \sqrt {\pi}}$$ and applying Lemma \ref{kryak}(i), we obtain (\ref{hnhn}).

{\rm (iii)}  The validity of this statement for  $\b> 2\!-\!i$ is a
consequence of Lemma \ref{kya} (with $m=2$) and Lemma
\ref{kryak}(ii) (with $m=1$). Consider the case $\b= 2\!-\!i$ which
is more subtle. Denote for short $F(t)=A_{i, \b} (t,\xi)$ and let
first  $\b> 1\!-\!i$. By Lemma \ref{kya} the derivative $F'(t)$ is
continuous in the neighborhood of $t=0$. Since  $F$ is an even
function, then $F'(0)=0$ and the left hand side of
 (\ref{hnhn}) can be written as
\be\label {nez} \frac{1}{\Gam (\a/2)}\int_0^\infty t^{\a-1} \Del
(t)\, dt, \quad \Del (t)=F(t)-F(0)-tF'(0). \ee By (\ref{kpl}) and
(\ref{kogd}),
$$
\Del (t)=\int_{ S^{N-1}\cap \xi^\perp} du \int_{ S^{N-1}\cap \xi}
\Del_{u,v} (t) dv
$$
where $\Del_{u,v} (t)=f(t)-f(0)-tf'(0)$,
$$
 f(t)\equiv a^\b_{u,v} (t)=\int_0^{\rho(tu,v)}\! r^{i-1} (r^2+t^2)^{\b/2}\, dr, \quad \rho\equiv \rho(tu,v).
$$
To estimate $\Del_{u,v} (t)$, we write it as $\Del_{u,v}
(t)=I_1+I_2$, where
$$
I_1=\int_0^{\rho(tu,v)}\! r^{i-1} [(r^2+t^2)^{\b/2}-r^\b]\, dr,
$$
$$
I_2=\int_0^{\rho(tu,v)}\! r^{i+\b-1}dr-\int_0^{\rho(0,v)}\!
r^{i+\b-1}dr- t[a_1 (0)+a_2(0)],$$
$$
a_1 (t)=\rho^{i-1}
(\rho^2+t^2)^{\b/2} d\rho/dt, \quad \rho\equiv \rho(tu,v), \quad a_2 (t)=\b t
a^{\b-2}_{u,v} (t).
$$
For $I_1$, changing the order of integration, we have $$
I_1=\frac{\b}{2}\int_0^{\rho}\! r^{i-1}dr\int_0^{t^2}
(r^2+s)^{\b/2-1}\, ds=\frac{\b}{4}\int_0^{t^2}s^{(i+\b)/2 -1}\,
h(s)ds, $$ $$ h(s)=\int_0^{\rho^2/s}\eta^{i/2 -1} (\eta +1)^{\b/2
-1}d\eta.$$ If  $\b= 2\!-\!i$ then $ h(s)=O(\log (1/s))$ as $s \to
0$ and therefore, $I_1=O(t^2 \log (1/t))$ as $t \to 0$.

To estimate $I_2$ we note that $a_2(0)=0$ (see (\ref{eda})) and
therefore, \bea I_2&=&\frac{1}{i+\b}\,
[\rho(tu,v)^{i+\b}-\rho(0,v)^{i+\b}-t (i+\b)
\rho(0,v)^{i+\b-1}\rho'(0,v)]\nonumber\\&=& \psi (t)- \psi (0)-
t\psi' (0), \quad  \psi (t)\equiv\rho(tu,v)^{i+\b}.\nonumber\eea
Hence, $I_2=O(t^2)$ as $t \to 0$. Since all estimates above are
uniform in $u$ and $v$, then the function $\Del (t)$ in (\ref{nez})
is $O(t^2 \log (1/t))$ as $t \to 0$. This enables us to extend this
integral by analyticity to all $Re \, \a
>-2$.

 The  statement (iv) follows from Lemma \ref{kya} (with $m=2$) and   (\ref {netl}).
\end{proof}

\section{Comparison of volumes. Proofs of the main results} We recall
basic notation related to Problem B. Let $K$ and $L$ be
origin-symmetric convex bodies  in $\rN$,  $N=dn$, where $n>1$, $d
\in \{1,2,4,8\}$;  $G$ is the class  (\ref{247}) of block diagonal
orthogonal transformations of $\bbr^{N}$, which includes the groups
$G_{\bbr}, \, G_{\bbc}, \, G_{\bbh, l}, G_{\bbh, r}$; see
(\ref{bl1})-(\ref{bl3}). The notation $\tilde \Gr_{N-d}
 (\bbr^N) $ is used for the respective manifolds (\ref {250}) of $(N-d)$-dimensional
 subspaces $H_\theta$, $ \theta \in S^{N-1}$, in particular, for
 $$ \Gr_{n-1} (\bbr^n), \quad \Gr_{2n-2}^{\bbc} (\bbr^{2n}),\quad \Gr_{4n-4}^{\bbh, l}
 (\bbr^{4n}),\quad \Gr_{4n-4}^{\bbh, r} (\bbr^{4n});$$
  see Section \ref{3344}.
 If $K$ is an infinitely smooth $G$-invariant star body in
$\rN$, then, by Lemma \ref{mal} and Corollary \ref {opya},
 \be \label {mige} S_K (\theta)\equiv vol_{N-d} (K \cap H_\theta )=c\,
 (M^{1-d}\rho_K^{N-d})(\theta),\ee
 \be \label{chto} (D_m
S_K)(\theta) =c\, (M^{1-d-2m}\rho_K^{N-d})(\theta), \qquad \ee where
$$c=\pi^{N/2 -d}\, \sig_{d-1}/(N-d), \quad (D_m f) (\theta)=2^{-2m}
[(-\Del)^m E_{-d} f](x)|_{x=\theta},$$ \be\label {alft} 2m \neq N-d,
N-d+2, N-d+4, \ldots .\ee

\begin{lemma}\label{35t} Let
\be\label {alftr} \a  \notin \{0, -2, -4, \ldots \} \cup \{N, N+2,
N+4, \ldots \}.\ee

\noindent {\rm (i)} If $K$ and $L$ are infinitely smooth
$G$-invariant star bodies in $\rN$ such that
 $(M^{\a +1-N}\rho_K^{d})(\theta)\ge 0$ and \be\label {zna}
 (M^{1-\a}\rho_K^{N-d})(\theta)\le  (M^{1-\a}\rho_L^{N-d})(\theta)\quad \forall \theta \in
 S^{N-1},\ee
then $vol_N(K) \le vol_N(L)$.

\noindent {\rm (ii)}  If $L$ is an infinitely smooth $G$-invariant
convex body with positive curvature such that
 $(M^{\a +1-N}\rho_L^{d})(\theta)<0$ for some  $ \theta \in S^{N-1}$, then there exists a
$G$-invariant smooth convex body $K$ for which (\ref{zna}) holds,
but $vol_N(K) > vol_N(L)$.
\end{lemma}
\begin{proof} {\rm (i)} By Lemma \ref{l1},
$$ N\,vol_N(K)=\int_{S^{N-1}} \rho_K^N
(\theta)\,d\theta=(\rho_K^{N-d}, \rho_K^{d})=(M^{1-\a}\rho_K^{N-d},
M^{\a +1-N}\rho_K^{d}).$$
 Since $M^{\a +1-N}\rho_K^{d}\ge 0$,  we can continue:
$$
N\,vol_N(K)\le (M^{1-\a}\rho_L^{N-d}, M^{\a +1-N}\rho_K^{d})=(\rho_L^{N-d},
\rho_K^{d}).$$ Now the result follows by H\"older's inequality.

{\rm (ii)} Let $\vp (\theta)\equiv (M^{\a
+1-N}\rho_L^{d})(\theta)<0$ for some $ \theta \in S^{N-1}$. Then
$\vp$ is negative on some open set $\Om \subset S^{N-1}$ and,  by
Lemma \ref{l1},  $\rho_L^d=M^{1-\a} \vp$. Since $\vp$ is
$G$-invariant, then $\vp <0$ on the whole orbit $G\Om$. Choose a
function $\psi \in
 \D(S^{N-1})$  so that $\psi \neq 0, \; \psi (\theta) >0$ if
 $ \theta\in G\Om$, and $\psi (\theta)\equiv 0$ otherwise. Without
 loss of generality, we can assume $\psi $ to be $G$-invariant (otherwise, it can be replaced
 by $\tilde\psi (\theta) =\int_G \psi (\gam\theta)\, d\gam$).  Define
 a smooth $G$-invariant body $K$ by $\rho_K^{N-d}=\rho_L^{N-d}-\e
 M^{\a +1-N}\psi$, $\e>0$. If $\e$ is small enough, then $K$ is convex.
 This conclusion is a consequence
of Oliker's formula \cite{Ol}, according to which the Gaussian
curvature of an origin-symmetric star body expresses through the
first and second derivatives of the radial function. Applying
$M^{1-\a}$ to the
 preceding equality, we obtain
 $$
M^{1-\a}\rho_K^{N-d}-M^{1-\a}\rho_L^{N-d}=-\e M^{1-\a}M^{\a
 +1-N}\psi=-\e \psi \le 0,$$
 which gives (\ref{zna}).  On the other hand,
 $$
 (\rho_L^{d}, \rho_L^{N-d} - \rho_K^{N-d})=\e (M^{1-\a}\vp, M^{\a
 +1-N}\psi)=\e(\vp, \psi)<0$$ or $(\rho_L^{d}, \rho_L^{N-d})<(\rho_L^{d}, \rho_K^{N-d})$.
 By H\"older's inequality, the latter implies
 $vol_N(L) < vol_N(K)$.
\end{proof}

Now, we investigate for which
 $\a$ the inequality   $(M^{\a +1-N}\rho_K^{d})(\theta)\ge 0$ in Lemma \ref{35t} is available.

\begin{lemma}  \label{main1} Let $K$ and $L$ be infinitely smooth
$G$-invariant convex bodies in $\rN$; $N=dn; \; n>1;\;d\in
\{1,2,4,8\}$. Suppose that  $$(M^{1-\a}\rho_K^{N-d})(\theta)\le
(M^{1-\a}\rho_L^{N-d})(\theta)\quad \forall \theta \in
 S^{N-1}$$
  for some $\a$ satisfying \be
\label{taic1} \max (N-d-2,d)\le \a <N.\ee
 Then  $vol_N(K) \le vol_N(L)$.
\end{lemma}
\begin{proof} We apply Lemma \ref {krya} with
$\xi\!=\!H_\theta$, $i\!=\!N\!-\!d$, and $\a$ replaced by $\a+d-N$.
By Lemma \ref {L22} the expression $(R_{N-i}M^{\a+1+i-N}\rho_K
^{\a+\b+i})(\xi^\perp)$ in Lemma \ref {krya} transforms into $
I_{\a,\b}=(M^{\a+1-N}\rho_K ^{\a+\b})(\theta)$ and the latter is
represented as follows.

\vskip 0.2truecm

\noindent $\bullet$ For $\a>N-d, \; \b>d-N$:

 \be \label{54}I_{\a,\b}=\frac{c^{-1}}{\Gam
((\a+d-N)/2)}\int_0^\infty t^{\a+d-N-1} A_{N-d, \b} (t,H_\theta)\,
dt.\ee

\noindent $\bullet$  For $\quad \a=N-d, \; \b>d-N$: \be \label{541}
I_{\a,\b}=\frac{1}{2}A_{N-d, \b} (0,H_\theta). \ee

\noindent $\bullet$  For (a) $N-d-1<\a<N-d, \quad 1+d-N<\b\le 0$, and\\
 ${}\qquad \:\, \;  $(b) $N-d-2<\a<N-d, \quad
2+d-N\le \b\le 0$: \bea \label{542} I_{\a,\b}&=&\frac{c^{-1}}{\Gam
((\a\!+\!d\!-\!N)/2)}\\
&\times&\int_0^\infty \!t^{\a+d-N-1}  [A_{N-d, \b}
(t,H_\theta)\!-\!A_{N-d, \b} (0,H_\theta)]\, dt.\nonumber\eea

\noindent $\bullet$ For $\a=N-d-2, \quad 2+d-N<\b\le 0$: \be
\label{543} I_{\a,\b}=-\frac{c_1^{-1}}{4}\,A''_{N-d, \b}
(0,H_\theta).\ee

\noindent Owing to Lemma \ref{kyah}, expressions
(\ref{54})-(\ref{543}) are nonnegative.  Set $\b=d-\a$ to get
$M^{\a+1-N}\rho_K^d \equiv I_{\a,d-\a}$. Then combine inequalities
in each case. We obtain the following bounds for $\a$.

 \noindent For $d=1$, $ N=n$:   $\max (n-3,1)\le \a<n$.

 \noindent For $d=2,4,8$:

 (\ref{54}) holds if $\quad N-d < \a<N$.

(\ref{541})  holds if    $\quad\a= N-d$.
 \bea \quad \mbox {\rm  (\ref{542}) holds if } \quad   N-d-1< \a<N-d\quad
&&\mbox{\rm when}\quad N
\ge 2d+1;\nonumber\\
 N-d-2 \le \a<N-d \quad &&\mbox{\rm when}\quad  N
\ge 2d+2;\nonumber\\
d\le \a<N-d\quad &&\mbox{\rm when}\quad   2d< N < 2d+2.\nonumber\eea

 (\ref{543}) holds if  $\quad \a=N-d-2, \quad  N \ge  2d+2$.\\
Combining these inequalities, we obtain (\ref{taic1}).
\end{proof}

\begin{remark} Operator $M^{1-\a}\equiv (M^{1+\a-N})^{-1}$ in Lemmas \ref{35t} and \ref{main1}, that was
originally defined  by analytic continuation of the integral
(\ref{af}), can be explicitly represented as an integro-differential
operator  $P(\del)M^\gam$, where $M^\gam, \; \gam>0$, has the form
(\ref{af}) and $P(\del)$ is a polynomial of the Beltrami-Laplace
operator $\del$ on $S^{N-1}$; see \cite[Section 2.2]{R2} for
details.
\end{remark}

Lemma \ref {main1} leads to   main results of the paper. The next
statement gives a positive answer to Problem {\bf B}.
\begin{theorem} \label {mcor}  Let $K$ and $L$ be  $G$-invariant
convex bodies in $\rN$  with section functions

\centerline{$S_K (\theta)= vol_{N-d} (K \cap H_\theta), \quad$
$\quad S_L (\theta)= vol_{N-d} (L \cap H_\theta)$,}

 \noindent where $H_\theta \in \tilde \Gr_{N-d}
 (\bbr^N)$, $N=dn, \; n>1,\;d\in
\{1,2,4,8\}$.  Suppose that \be\label{bp4} S_K (\theta)\le S_L
 (\theta) \qquad
\forall \theta \in S^{N-1}.\ee If $n\le 2+2/d$, then  $vol_N(K) \le
vol_N(L)$.
\end{theorem}
\begin{proof}  For infinitely smooth bodies the result is contained in Lemma \ref
{main1} (set $\a=d$ and make use of (\ref{mige})). Let us  extend
this result to arbitrary $G$-invariant convex bodies. Given a
$G$-invariant convex body $K$, let $$K^*=\{x: |x\cdot y| \le 1\;
\forall y\in K \}$$ be the polar body of  $K$ with  support function
$$h_{K^*} (x)=\max \{ x\cdot y : y \in K^*\}.$$ Since $h_{K^*}
(\cdot)$ coincides with Minkowski's functional $ ||\cdot||_K$, then
$h_{K^*} (\cdot)$ is $G$-invariant, and therefore, $K^*$ is
$G$-invariant too. It is known \cite [pp. 158-161] {Schn}, that any
origin-symmetric convex body in $\rN$ can be approximated   by
infinitely smooth convex bodies with positive curvature and the
approximating operator commutes with rigid motions. Hence, there is
a sequence $\{K_j^*\}$ of infinitely smooth $G$-invariant convex
bodies with positive curvature such that $h_{K^*_j} (\theta)$
converges to $h_{K^*}(\theta)$ uniformly on $S^{N-1}$. The latter
means, that for the relevant sequence of infinitely smooth
$G$-invariant convex bodies $K_j=(K_j^*)^*$ we have
$$\lim_{j \to \infty }\,\max_{\theta \in S^{N-1}} |\, ||\theta||_{K_j}-
||\theta||_K |=0.$$
 This implies convergence in the radial metric, i.e.,
\be\label{conv1}\lim_{j \to \infty }\,\max_{\theta \in S^{N-1}}
|\rho_{K_j}(\theta) - \rho_{K}(\theta)|=0.\ee

Let us show that the sequence $\{K_j\}$ in (\ref{conv1}) can be
modified so that $K_j \subset K$. An idea of the argument was
borrowed from \cite {RZ}. Without loss of generality, assume that
$\rho_K(\theta)\ge1$. Choose $K_j$ so that
$$
\vert \rho_{K_j}(\theta)-\rho_K(\theta)\vert < \frac1{j+1} \quad
\forall \theta \in S^{N-1}
$$
and set  $K_j'=\frac j{j+1} K_j$. Then, obviously,
$\rho_{K_j'}(\theta) \to \rho_K(\theta)$ uniformly on $S^{N-1}$ as
$j \to \infty$, and
$$
\rho_{K_j'}=\frac j{j+1} \rho_{K_j} <\frac j{j+1} \big
(\rho_K+\frac1{j+1}\big )\le \rho_K.
$$
Hence, $K_j'\subset K$. Now suppose that (\ref{bp4}) is true. Then
it is true when $K$ is replaced by $K_j'$, and, by the assumption of
the lemma, $\vol_N(K_j') \le \vol_N(L)$. Passing to the limit as $j
\to \infty$, we obtain $\vol_N(K) \le \vol_N(L)$.
\end{proof}

The following theorem, which generalizes Theorem 4 from \cite{KKZ},
shows that the restriction $n \le 2+2/d$ in Theorem \ref{mcor} is
sharp.
\begin{theorem}\label{37} Let $N=dn > 2d+2,\; n>1,\;d\in
\{1,2,4,8\}$.  Then there exist $G$-invariant infinitely smooth
convex bodies $K$ and $L$ in $\bbr^{N}$ such that  $ S_K(\theta)\le
S_L(\theta)$ for all $ \theta \in S^{N-1}$, but $vol_N(K) >
vol_N(L)$.
\end{theorem}
\begin{proof}
Let  $x =(x_1, \ldots, x_n)^T \in \bbr^{N}, \; x_j=( x_{j,1} \ldots,
x_{j,d})^T$,
$$
L =\{ x : ||x||_4= \Big (\sum\limits_{j=1}^n |x_j|^4 \Big )^{1/4}
\le 1 \}.
$$ Clearly, $L$  is  a $G$-invariant
infinitely smooth convex body. Let $X$ be the $(N-d+1)$-dimensional
subspace of $\bbr^{N}$, which consists of vectors of the form
$(x_{1,1}, x_2, \dots , x_n)^T$. By \cite [Theorems 4.19, 4.21]{K},
$L
 \cap X$ is not a $\lam$-intersection body in $\bbr^{N-d+1}$ if  $0<\lam <
 N-d-2$. Hence, by Theorem \ref{tyif}, $L$ is not a $\lam$-intersection body  for such
 $\lam$. It means (see Definition \ref {dk}) that $(M^{1+\lam
 -N}\rho_K^\lam)(\theta)<0$ for some  $ \theta \in S^{N-1}$.  Set
 $\lam=d$ to get $dn > 2d+2$ and apply Lemma \ref{35t}(ii) with $\a=d$.
 This gives the result.
\end{proof}

\begin{corollary} \label {krrr} The Busemann-Petty problem {\bf A} in $\kn, \; n>1,$ has an affirmative
answer if and only if $n \le 2+2/d$. In particular,

in $\rn$:  if and only if $n \le 4$;

in $\cn$:  if and only if $n \le 3$;

in $\bbh^n_l$ and $ \bbh^n_r$:  if and only if $n = 2$.
\end{corollary}

Theorem \ref {mcor} also implies the following.
\begin{corollary} \label {mcors} Let $d\in \{2,4,8\}$, $i=N-d$. The lower dimensional
Busemann-Petty problem for $i$-dimensional sections of
$N$-dimensional $G$-invariant convex bodies has an affirmative
answer in the following cases:
\[
\begin{array}{llll}
(a) & N=4 & (d=2): & i=2, \\
(b) & N=6 & (d=2): & i=4, \\
(c) & N=8 & (d=4): & i=4, \\
(d) & N=10 & (d=4): & i=6, \\
(e) & N=16 & (d=8): & i=8. \\
\end{array}\]
\end{corollary}

Another consequence of Lemma \ref {main1}, which addresses Problem
{\bf C}, can be obtained if we set $\a=d+2m$ in that Lemma and make
use of Corollary \ref{opya}.

\begin{theorem}\label{37b} Let $K$ and $L$ be infinitely smooth
$G$-invariant convex bodies in $\rN$; $N=dn, \; n>1,\;d\in
\{1,2,4,8\}$. Suppose that  $$(-\Del)^m E_{-d}S_K(\theta)\le
(-\Del)^m E_{-d}S_L(\theta) \quad \forall \theta \in
 S^{N-1}$$
  for some $m$ satisfying \be
\label{taic3} \max (N-2d-2,0)\le 2m <N-d.\ee
 Then  $vol_N(K) \le vol_N(L)$. In particular, $m$ can be
 chosen as follows:
\bea
 \text{For $d\!=\!1$:} &m&\!\!\!=0\; \text{ if $n \!\le \!4$, and}\;m\!\in \!\left\{\frac{n\!-\!4}{2},
 \,\frac{n\!-\!3}{2},  \,\frac{n\!-\!2}{2}\right\} \; \text{if $\;n \!> \!4$}.\nonumber\\
 \text{For $d\!=\!2$:} &m&\!\!\!=0\; \text{ if $n \le 3$, and $m \!\in\!
\{n-3, \; n-2\}$ if $\;n >
 3$.}\nonumber\\
 \text{For $d\!=\!4$:} &m&\!\!\!=0\; \text{ if $n =2$, and $m \!\in \!\{2n\!-\!5, \; 2n\!-\!4, \; 2n\!-\!3\}$ if $n >
 2$.}\nonumber\\
 \text{For $d\!=\!8$:} &m&\!\!\!=0\; \text{ if $n =2$, and} \nonumber\\
 &{}&\!\!\!\!\!\!\!\!\!\!\text{ $m \!\in \!\{4n\!-\!9, \; 4n\!-\!8, \; 4n\!-\!7, \; 4n\!-\!6, \; 4n\!-\!5, \;\}$ if $n >
 2$.}\nonumber \eea
\end{theorem}

\section {Appendix: Proof of Lemma \ref{lev}}

 {\rm (i)}  We recall (see Definition
 \ref{levd}) that a function $p: V \to \bbr$ is a norm if the
following conditions are satisfied:

{\rm (a)} $p(x) \ge 0$ for all $x \in V$;  $\;p(x) =0$ if and only
if $x=0$;

{\rm (b)} $p(\lam x)=|\lam |p(x)$  for all $x \in V$ and all $ \lam
\in \frA$;

{\rm (c)} $p(x+y)\le p(x) + p(y)$  for all $x,y \in V$.

\noindent If $V$ is a right  space over $\frA$, then {\rm (b)} is
replaced by

{\rm (b$'$)} $p(x\lam)=|\lam |p(x)$  for all $x \in V$ and all $
\lam \in \frA$.

  Let $V$ be a left  space (for the right
space the argument follows the same lines with (b) replaced by
 (b$'$)). Suppose that $p: V\to \bbr$ is a norm and show that \be A_p =\{x \in V: p(x)
\le 1\}\ee is an equilibrated convex body. Let
 $x, y \in A_p$. Then for any nonnegative $\a$ and $ \b$
satisfying $\a+\b=1$, owing to (b) and (c), we have
$$
p(\a x+\b y) \le p(\a x) +p(\b y)=\a p( x) +\b p(y)\le \a+\b=1.$$
Hence, $\a x+\b y \in A_p$, that is, $A_p$ is convex. Since for
every $\lam \in \frA$ with $|\lam|\le 1$, (b) implies $p(\lam x)
=|\lam| p(x) \le 1$, then $\lam x\in A_p$. Thus $A_p$ is
equilibrated. To prove that $A_p$ is a body, it suffices to show
that $A_p$ is compact and the origin is an interior point of $A_p$.
To this end, we first prove that $p$ is a continuous function. Let
$x=x_1 f_1 + \ldots x_n f_n$, as above. By (b) and (c), \bea p(x)
&\le& p(x_1 f_1) + \ldots +p(x_n f_n) =|x_1| p( f_1) + \ldots +
|x_n| p( f_n)\nonumber\\&\le& \gam \sum\limits_{j=1}^n |x_j|, \qquad
\gam =\max_{j=1, \ldots ,n} p(f_j).\nonumber\eea Now for $x=(x_1,
\ldots , x_n)$ and $y=(y_1, \ldots , y_n)$, owing to (c), we have
$$ p(x) \le p(y)+ p(x-y), \qquad p(y) \le p(x)+ p(y-x)=p(x)+
p(x-y).$$ Hence, $$  |p(x)-p(y)|\le p(x-y)\le
\gam\sum\limits_{j=1}^n |x_j - y_j|,$$ and the continuity of $p$
follows. Furthermore, since $p(x) >0$ for every $x$ on the unit
sphere $\Om=\{x \in V : ||x||_2 =1 \}$ and since $p$ is continuous,
there exists $\del >0$ such that $p(x) >\del$ for all $x \in \Om$.
If $x \in A_p$ and $x'=x/||x||_2  \in \Om$, then $ 1 \ge
p(x)=||x||_2  p(x')>\del ||x||_2$, i.e., $||x||_2<\del ^{-1}$. Thus,
$A_p$ is bounded. Since $A_p$ is also closed as the inverse image of
the closed set $0\le\lam \le 1$, it is compact.

To prove that $A_p$ is a body, it remains to show that $A_p$
contains the origin in its interior. Since $p$ is continuous and
$\Om$ is compact, there is a number $\b>0$ such that $p(x') \le \b$
for all $x'  \in \Om$. Then the open ball $B_{1/\b} =\{ x \in V:
||x||_2<1/\b \}$ lies in  $A_p$, because for $x \in B_{1/\b}$,
$p(x)=||x||_2 p(x') \le ||x||_2 \b <1$.

(ii) Suppose that $A \subset V$ is an equilibrated convex body and
let us prove (a)-(c) for $p_A (x)=\inf \{ r>0: x \in r A \}$. Since
$A$ is equilibrated, then $0 \in A$ and therefore, $p_A (0)=\inf \{
r>0: 0 \in r A \}=0$. Conversely, if  $p_A(x) \equiv \inf \{ r>0: x
\in r A \}=0$, then for every $k \in \bbn$, there exists $r_k <1/k$
such that $x \in r_k A$. Since $A$ is equilibrated, then $r_k A$ is
equilibrated too, thanks to the following implications that hold for
all $\lam \in \bbk, \; |\lam |\le 1$:
$$
x \in r_k A \Longrightarrow \frac{x}{ r_k}\in A \Longrightarrow
\frac{\lam
 x}{ r_k}\in A \Longrightarrow \lam  x \in r_k A.$$
 Since $r_k A$ is equilibrated, then  $0\in r_k A$ for all $k$.
 Passing in $x\in r_k A$ to the limit as $k \to \infty$, we get $x=0$.
 This gives (a).

 Let us check (b). For $\lam =0$, (b) follows from (a). Let  $\lam \neq
 0$.  Since $A$ is equilibrated, then for every $r>0$, $ \lam x\in rA$
 if and only if $  x\in \frac{r}{ |\lam |} A$. Hence, \bea p_A (\lam x)
 &=&\inf \{r>0:  \lam x\in rA \}=\inf \{r>0:  x\in \frac{r}{|\lam |} A
 \}\nonumber \\&=& |\lam |\inf \{r>0:  x\in rA \}=|\lam |  p_A ( x).\nonumber
 \eea

 To prove (c), choose $\a, \b >0$ and let $x\in \a A, \;y\in \b A$.
 Then $$ x+y =(\a+\b)\left ( \frac {\a}{\a+\b}\,\frac{x}{\a}+
\frac {\b}{\a+\b}\,\frac{y}{\b}\right ).$$
 Since the points $\a^{-1} x$ and $\b^{-1} y$ are in $A$ and $A$ is
 convex, the weighted sum in parentheses is also in $A$, and therefore,
$ x+y \in (\a+\b) A$. This gives $p_A (x+y) \le \a+\b$. By letting
$\a=p_A(x), \;\b=p_A (y)$, we are done. \hfill $\square$

\end{document}